\documentclass[11pt, twoside, reqno]{amsart}
\usepackage[a4paper,left=35mm,right=35mm,top=30mm,bottom=30mm,marginpar=25mm]{geometry} 

\usepackage{xcolor}    
\usepackage{hyperref}
\definecolor{darkcyan}{cmyk}{1, 0, 0, 0.6}
\hypersetup{
    colorlinks=true,
    linkcolor=darkcyan,  
    urlcolor=darkcyan,    
    citecolor=darkcyan 
}
    \usepackage{amsmath}
\usepackage{enumerate}
\usepackage{tikz}
\usepackage{comment}
\usepackage{bbm}
\usepackage{amssymb}
\usepackage{mathrsfs}
\usepackage{stmaryrd}
\usepackage{enumitem}
\usepackage{dirtytalk}
\usepackage{tikz-cd} 
\newtheorem{thm}{Theorem}[section]
\newtheorem{prop}[thm]{Proposition}
\newtheorem{coro}[thm]{Corollary}
\newtheorem{lemma}[thm]{Lemma}
\theoremstyle{definition}
\newtheorem{exmp}[thm]{Example}
\newtheorem{defi}[thm]{Definition}
\theoremstyle{remark}
\newtheorem{remark}[thm]{Remark}
\numberwithin{equation}{section}

\newcommand*\dif{\mathop{}\!\mathrm{d}}

\newcommand{\second}{{\prime \prime}}
\newcommand{\R}{\mathbb R}

\newcommand{\mc}[1]{\mathcal{#1}}

\newcommand{\mr}[1]{\mathrm{#1}}
\newcommand{\bs}[1]{\boldsymbol{#1}}
\newcommand{\ms}[1]{\mathsf{#1}}
\newcommand{\msc}[1]{\mathscr{#1}}

\newcommand{\ol}[1]{\overline{#1}}
\newcommand{\wt}[1]{\widetilde{#1}}

\author{Fabio Ancona}
\address{Dipartimento di Matematica \say{Tullio Levi Civita}, Università di Padova}
\email{ancona@math.unipd.it  \textnormal{(Fabio Ancona)}}
\author{Elio Marconi}
\address{Dipartimento di Matematica \say{Tullio Levi Civita}, Università di Padova}
\email{elio.marconi@unipd.it  \textnormal{(Elio Marconi)}}
\author{Luca Talamini}
\address{Mathematics Area, SISSA, Trieste}
\email{ltalamin@sissa.it  \textnormal{(Luca Talamini)}}

\title[Quasi-entropy solution to scalar balance laws]
{On the structure of entropy dissipation and regularity for quasi-entropy solutions \\ to 1d scalar conservation laws \\ and to isentropic Euler  system with $\gamma=3$}

\begin{document}

\begin{abstract}
In this paper, we first investigate quasi-entropy solutions to scalar conservation laws in several space dimensions. In this setting, we introduce a suitable Lagrangian representation for such solutions. Next, we prove that, in one space dimension and for fluxes $f$ satisfying a general non-degeneracy condition, the entropy dissipation measures of quasi-entropy solutions are concentrated on a $1$-rectifiable set. The same result is obtained for the isentropic Euler system with $\gamma = 3$, for which we also slightly improve the available fractional regularity by exploiting the sign of the kinetic measures. 
\end{abstract}
\maketitle

\section{Introduction}
Consider a scalar conservation law
\begin{equation}\label{eq:conslaw}
    \partial_ t u + \mr{div}_x f(u) = 0,  \qquad \text{in} \quad (0, T) \times \mathbb R^d
\end{equation}
where $f: \mathbb R \to \mathbb R^d$ is a $C^2$ function. In this paper we study \emph{quasi-entropy solutions} of \eqref{eq:conslaw}. These are functions $u \in \mathbf L^{\infty}((0, T) \times \mathbb R^d)$ such that the distribution $\mu_\eta := \eta_t(u) + q_x(u)$ is a locally finite measure for every entropy-entropy flux $(\eta, \bs q)$, i.e. a pair of $C^2$ smooth functions such that $\bs q^\prime = f^\prime \eta^\prime$. In general, if $u$ is a quasi-entropy solution to \eqref{eq:conslaw}, it satisfies the balance law of the form
\begin{equation}\label{eq:scamultsource}
\begin{aligned}
    & \partial_t  u + \mathrm{div}_x f(u) = m, \qquad \text{in} \quad (0, T) \times \mathbb R^d.
    \end{aligned}
\end{equation}
where $m \in \msc M([0,T] \times \mathbb R^d)$, the source, is a locally finite measure. 
It is well known that weak solutions of \eqref{eq:conslaw} are not unique, and that the well posedness is restored only in the smaller class of \textit{entropy solutions}. These last are weak solutions to \eqref{eq:conslaw} that in addition dissipates every convex entropy:
\begin{equation}\label{eq:entropy}
    \mu_{\eta} \doteq \partial_ t \eta(u) + \mr{div}_x \bs q(u) \leq 0, \qquad \text{in $\msc D^{\prime}$}.
\end{equation}
where here $\eta$ (the entropy) is any scalar convex function.
Since \cite{Kr77} it is known that the Cauchy problem associated to \eqref{eq:conslaw} is well posed in the class of entropy solutions.

In the fundamental paper \cite{LPT_kinetic} the authors prove that entropy solutions of \eqref{eq:conslaw} satisfy a \textit{kinetic formulation}, where the collision term, which encodes the entropy production \eqref{eq:entropy}, takes the form of a derivative of a nonnegative finite measure:
\begin{equation}\label{eq:kineticintro}
\partial_t \chi + f^{\prime}(v) \cdot \nabla_x \chi = \partial_v \mu, \qquad \chi(t, x, v) = \mathbf{1}_{0 < v \leq u(t, x)}-\mathbf{1}_{u(t, x) \leq v < 0} .
\end{equation}

Although entropy solutions provide a satisfactory mathematical theory for \eqref{eq:conslaw}, more general concepts of solution are studied in the literature, in particular the class of \textit{finite entropy solutions} (or, solutions with finite entropy production) appears in many physical models: these are quasi-entropy solutions for which, in addition, $m = 0$ in \eqref{eq:scamultsource}. Being a finite entropy solution corresponds, from the kinetic formulation point of view, to allowing finite measures $\mu$ (without sign) in the right hand side of \eqref{eq:kineticintro}. For quasi-entropy solutions, an additional source term $\mu_0 \in \msc M([0,T] \times \mathbb R^d \times \mathbb R)$ must be added:
\begin{equation}\label{eq:kinmu0intro}
    \partial_t \chi + f^{\prime}(v) \cdot \nabla_x \chi = \partial_v \mu_1 + \mu_0
\end{equation}

Finite entropy solutions (therefore, quasi-entropy solutions satisfying \eqref{eq:kinmu0intro} with $\mu_0 = 0$) arise for example in the study of large deviations for scalar conservation laws \cite{bertini_al}, \cite{Mariani_large} and in the study of the $\Gamma$-convergence problem of the Aviles-Giga functional, see e.g. \cite{marconi_aviles} and the references therein. Moreover, finite entropy solutions find applications in the theory of diffusive-dispersive approximations to conservation laws and more generally whenever non-classical shock waves are present, see \cite{LeFloch2002} and the references therein.
 In \cite{HT02}   it is proved that diffusive dispersive approximations to conservation laws with general flux functions converge to finite entropy solutions, where $\mu_1$ is a general locally finite measure. It is expected that the results of this paper will be useful in the study of these solutions. More importantly for our present purposes, as it shown in Proposition \ref{prop:wzfes} later in this paper, the Riemann invariants of an entropy solution of the Euler system with $\gamma = 3$ 
\begin{equation}\label{eq:isointro}
\begin{aligned}
    & \partial_t \rho + \partial_x (\rho u) = 0, \\
    & \partial_t (\rho u)  + \partial_x (\rho u^2 + \rho^3/3) = 0
    \end{aligned}
\end{equation}
are quasi-entropy solutions to the Burgers equation
\begin{equation}\label{eq:burintro}
    \partial_t u + \partial_x \left( \frac{u^2}{2}\right) = 0
\end{equation}
if $\rho$ is bounded away from zero. This explains the need for a source term $\mu_0$ in \eqref{eq:kinmu0intro}. The system has been studied in the literature by various authors, see for example \cite{Gol23,LPT_kineticise, vasseur_gamma3,  Vas99}. In \cite{ABBCN} it is conjectured that systems of conservation laws might share some of their regularity properties with scalar balance laws. In this direction, in \cite{AMT25} it is proved that entropy solutions to $2\times 2$ systems satisfy a pair of kinetic equations with source terms, and it is proved that points outside the jump set are of \emph{vanishing mean oscillation}. Moreover, in \cite{Tal25}, these tools are further exploited to obtain strong time regularity and a decay result for entropy solutions.

In one space dimension, starting from the work of Ole{\u\i}nik \cite{Oleinik_translation}, various authors studied the regularizing effect induced by the interaction of the nonlinearity of the flux with the entropy condition \eqref{eq:entropy}. Typically, the nonlinearity of the flux forces \textit{characteristic lines} to intersect; together with the entropy condition this produces the desired regularizing effect. For example, for entropic solutions $u$ to Burgers equation, the one-sided Lipschitz estimate 
$$
\partial_x u(t, \cdot) \leq \frac{1}{t}, \qquad \text{in $\msc D^{\prime}$}
$$
induces a regularizing effect $L^{\infty}$ to $BV$. See also   \cite{ABM25, AT26, BM_structure, Junca, LPT_kineticise, LPT_kinetic, Mar_reg} for various results on the regularizing effect of nonlinearity in conservation law.  Therefore, from Vol'pert chain rule for $BV$ functions \cite{volpert_chain}, it follows that the measures $\mu_{\eta}$ are concentrated on the 1-rectifiable jump set of $u$. It has been a long standing problem to understand whether or not this concentration property holds for general fluxes and entropy solutions in one space dimension. In one space dimension, a first result \cite{DLR03} solves the problem for entropy solutions with $C^2$ fluxes for which $\{u | f^\second(u) = 0 \}$ is locally finite, and only in a recent work \cite{BM_structure} the authors prove the same concentration property for entropy solutions with general smooth fluxes. The result of \cite{BM_structure} is achieved by using in a crucial way a \textit{Lagrangian representation} for solutions of \eqref{eq:conslaw}, which is an extension of the method of characteristics in the non-smooth setting.

In \cite{marconi_burgers} it is proved that the concentration property remains valid for finite entropy solutions to one dimensional Burgers equation. The result is also achieved by using a Lagrangian representation, but a different one: one can think of it as representing characteristics at the level of the kinetic formulation \eqref{eq:kineticintro}, rather then at the level of \eqref{eq:conslaw}. This kind of representation was introduced for the first time in \cite{BBM_multid} for entropic (multi-d) solutions of \eqref{eq:conslaw} and in \cite{marconi_structure} for finite entropy solutions.

In the multidimensional case, using techniques coming from geometric measure theory, in \cite{DOW03} a partial rectifiability result for the entropy dissipation is proved. In particular, letting $\nu(t, x)$ be the $(t, x)$-marginal of the measure $|\mu|$, with $\mu$ as in \eqref{eq:kineticintro}, they prove that the set $\mathbf J \subset (0,T) \times \mathbb R^d$  of points $(t, x)$ of positive $\msc H^{d}$ density of $\nu$ is $d$-rectifiable.  However, the question of whether $\nu$ is concentrated on the jump set $\mathbf J$ remains open.

\subsection{Results and plan of the paper} The  paper is organized as follows.

In Section \ref{sec:KFFES} we discuss the equivalence between the notion of quasi-entropy solution and the kinetic formulation \eqref{eq:kinmu0intro}. The results of this section are quite standard, see for example \cite{DOW03}.

In Section \ref{sec:lagrangian} we discuss a further equivalence, the one between quasi-entropy solutions and solutions which satisfy a Lagrangian representation, in the spirit of \cite{BBM_multid, marconi_structure}.  The main theorem, Theorem \ref{thm:lagrkin}, relies on a theorem of Smirnov \cite{smirnov_currents} (see Theorem \ref{thm:smirnov}) and we prove the existence of the Lagrangian representation via a measure-theoretic reparametrization argument. In this way we can also show some additional properties of the Lagrangian representation (Theorem \ref{thm:lagraex}).

In Section \ref{sec:structurekm} we study the structure of the dissipation measures for scalar conservation laws in one space dimension, satisfying a general non-degeneracy condition on the flux $f$, Definition \ref{def:fluxgenn}. In Theorem \ref{thm:1drect}, which is the main Theorem of this section, we show that the entropy dissipation measure (which is the pushforward measure $(p_{t,x})_\sharp |\mu_1|$ of $|\mu_1|$ by the canonical projection in the $t$,$x$ variables) is concentrated on a 1-rectifiable set. The arguments are inspired by \cite{marconi_burgers}, \cite{Mar23}.

In Section \ref{sec:RBE}, Theorem \ref{thm:besov}, we prove a regularity result for quasi entropy solutions to the Burgers equation when the measure $\mu_1$ in \eqref{eq:kinmu0intro} has a positive sign. The argument relies on the Lagrangian representation of Section \ref{sec:lagrangian}, and exploits the sign of the kinetic measure $\mu_1$, slightly improving the fractional regularity available in the literature \cite{GP13} in terms of Besov spaces.

Finally, in Section \ref{sec:isentropic} we show that the all the above results can be applied to entropy solution away from vacuum of the Euler system with $\gamma = 3$.

\section{Kinetic Formulation of quasi-entropy solutions}\label{sec:KFFES}
We start by giving the definition of quasi-entropy solution. Here we consider a $C^2$ \textit{flux} function $f  : \mathbb R \to \mathbb R^d$ and an open set $\Omega \subset \mathbb R^{d+1}$. Most of the times we will take $\Omega = (0,T) \times \mathbb R$ for some $T >0$.

\begin{defi}\label{defi:fes}
    We say that $u \in \mathbf L^\infty(\Omega; \mathbb R)$ is a \textit{quasi-entropy} solution with flux $f$ if for every entropy-entropy flux pair $(\eta, \bs q)$, with $\eta \in C^2(\mathbb R)$, there exists a locally finite measure $\mu_{\eta} \in \msc M(\Omega)$ such that 
    \begin{equation}\label{eq:fenergy}
        \partial_t \eta(u) + \mr{div}_x \, \bs q(u) \, = \, \mu_{\eta} \qquad \text{in $\msc D^\prime_{t,x}(\Omega)$}.
    \end{equation}
\end{defi}
Since in particular the pair $(\eta_0, \bs q_0)$ defined by 
$
\eta_0(u) = u$, $\bs q_0(u) = f(u)
$
is an entropy-entropy flux pair, if $u$ is a quasi-entropy solution we find that $u$ solves \eqref{eq:scamultsource} with $m = \mu_{\eta_0}$.
\begin{remark}
 The more classical concept of entropy solution to the conservation law
 $$
\partial_t u + \mr{div}_x \, f(u) = 0
$$
can be obtained by requiring, in addition, that $\mu_{\eta_0} = 0$ and that for every \textit{convex} entropy $\eta$, the corresponding entropy measure $\mu_{\eta}$ is nonpositive.
\end{remark}

 In \cite{LPT_kinetic}, Lions, Perthame and Tadmor characterize entropy solutions via a \textit{kinetic formulation}. 
Here we use a slightly different formulation, that immediately follows from the one of \cite{DOW03}, which is adapted for quasi-entropy solutions. We write its proof for completeness.
\begin{prop}\label{prop:feskin}
    A function $u \in \mathbf L^\infty(\Omega; \mathbb R)$ is a quasi-entropy solution in the sense of Definition \ref{defi:fes} if and only if, 
there exist locally finite measures $\mu_0, \mu_1 \in \msc M(\Omega \times \mathbb R)$ with $\mr{supp}\, \mu_i \subset \Omega \times [\inf u, \sup u]$, such that 
\begin{equation}\label{eq:kinu}
    \partial_t \chi + f^\prime(v) \cdot \nabla_x \chi = \partial_v \mu_1 + \mu_0 \qquad \text{in $\msc D^\prime(\Omega \times \mathbb R)$}
\end{equation}
where 
\begin{equation}\label{eq:chidef}
\chi(t, x, v) = \begin{cases}
    1 & \text{if $v \leq u(t,x)$},\\
    0 & \text{otherwise}.
\end{cases}
\end{equation}
\end{prop}
\begin{remark}
    Clearly the measures $\mu_0, \mu_1$ such that \eqref{eq:kinu} holds are not unique. It is possible to choose the measures $\mu_1, \mu_0$ such that their support is contained in the topological boundary of the hypograph of $u$:
    $$
    \mr{supp} \, \mu_i \subset \partial \, \mr{hyp} \, u \qquad i = 0,1,
    $$
    where
    $$
    \mr{hyp}(u)=\big\{(t,x,v)\in\Omega\times \R\,|\,
    v\leq u(t,x)\big\}.
    $$
    However, later on we will need to make a finer choice on the measures $\mu_0, \mu_1$, that in particular implies the above condition.
\end{remark}
 \begin{proof}
     We first show that if $u$ satisfies \eqref{eq:kinu}, then $u$ is a quasi-entropy solution. Let $(\eta, \bs q)$ be any entropy-entropy flux pair. Let $U \Subset \Omega$ (compactly contained in $\Omega$) and for every test function $\phi \in C^{\infty}_c(\Omega)$, with $\mr {supp} \, \phi \subset U$, we can compute, omitting the variables $(t,x)$, 
     $$
     \begin{aligned}
         \int_{\Omega} \phi_t \,\eta(u) + \nabla_x \phi \cdot \bs q(u) \dif x \dif t &  = \int_{\Omega} \int_{\mathbb R} \chi(v) \,\Big( \phi_t\eta^\prime(v)+ \nabla_x \phi\, \eta^\prime(v) \ f^\prime(v)\Big)  \dif v \dif x \dif t\\
         & = -\int_{\Omega} \int_{\mathbb R} \phi \, \eta^\prime \, \dif \mu_0(t, x, v) + \int_{\Omega} \int_{\mathbb R} \phi \, \eta^{\second} \, \dif \mu_1(t,x, v) \dif x \dif t.
     \end{aligned}
     $$
 In particular, letting $C_U = \Vert\mu_0\Vert_{\msc M( U)} + \Vert \mu_1\Vert_{\msc M(U)}$ and defining the distribution
 $$
 T_{\eta} \doteq \partial_t \eta(u) + \nabla_x \bs q(u)
 $$
 we obtain the bound
 \begin{equation}\label{eq:bTeta}
   \langle T_\eta, \, \phi\rangle  \leq C_U\, \Vert \phi\Vert_{\mc C^0} \, \Vert \eta^\prime\Vert_{\mc C^1} \qquad \forall \; \phi \in C^\infty_c(\Omega) \qquad \text{with $\mr{supp}\, \phi \subset U$}.
 \end{equation}
 By Riesz theorem, $T_{\eta}$ can be identified with a locally  finite measure.

 Conversely, assume that $u$ is a quasi entropy solution as in Definition \ref{defi:fes}. Define a distribution $T \in \msc D^\prime(\Omega \times \mathbb R)$ by 
 $$
 \langle T, \phi\, \varrho \rangle \doteq \int_{\Omega} \phi_t \eta_{\varrho}(u)\,  +\, \nabla_x \phi\cdot   \bs q_\varrho(u) \dif x \dif t \qquad \forall \phi \in C^\infty_c(\Omega), \quad \varrho \in C^\infty_c(\mathbb R)
 $$
 where we define the entropy-entropy flux pair associated to $\varrho$
 $$
 \eta_{\varrho}(v) \doteq \int_0^v \int_0^z \varrho(w) \dif w \dif z, \qquad \bs q_\varrho(v) \doteq \int_0^v f^\prime(z) \eta^\prime_\varrho(z)  \dif z.
 $$
 This definition is sufficient because finite sums $\sum_{i=1}^N \varphi_i(t,x) \varrho_i(\xi)$ are dense in $C_c^{\infty}(\mathbb R^3)$ (see e.g. \cite[Section 4.3]{FJ99})
We again consider any $U \Subset \Omega$ and for any $\phi \in \msc D(U)$ we define a linear functional $L_\phi \, : \,  C(\mathbb R) \to \mathbb R$ by 
$$
L_\phi(\varrho) \doteq   \int_{U} \phi_t \eta_{\varrho}(u)\,  +\, \nabla_x \phi\cdot   \bs q_\varrho(u) \dif x \dif t.
$$
Each functional $L_\phi$ is bounded, and therefore also continuous, since it holds
$$
|L_\phi(\varrho)| \leq C_{U, \phi} \, \Vert \varrho\Vert_{\mc C^0} \qquad \forall \varrho \in  C(\mathbb R)
$$
for some constant $C_{U, \phi}$ depending only the set $U$ and the $C^1$ norm of the function $\phi$.
Since $u$ is a quasi-entropy solution, we deduce that the family  of functionals $L_\phi$ is pointwisely bounded, because
$$
\sup_{\substack{\phi \in C^\infty_c(U) \\ |\phi|_0 \leq 1} } |L_\phi(\varrho)| \leq \int_U \dif |\mu_{\eta_\varrho}|.
$$
Therefore, by the uniform boundedness principle, the family $L_\phi$ is uniformly (norm) bounded, that is
\begin{equation}\label{eq:ubp}
    \sup_{\substack{\phi \in C^\infty_c(U) \\ |\phi|_0 \leq 1, \,|\varrho|_0 \leq 1} } |L_\phi(\varrho)| = \sup_{\substack{\phi \in C^\infty_c(U) \\ |\phi|_0 \leq 1, \,|\varrho|_0 \leq 1} } |\langle T, \phi \varrho\rangle|  \leq C_U.
\end{equation}
We also notice that by definition of $L_\phi$, it holds
\begin{equation}\label{eq:suppmu1}
L_\phi(\varrho)  = 0 \qquad \forall \varrho \in  C(\mathbb R) \quad \text{with $\mr{supp}\, \varrho \subset \mathbb R \setminus [\inf u, \sup u]$}.
\end{equation}
Combining \eqref{eq:ubp}, \eqref{eq:suppmu1} with the Riesz representation theorem, we then obtain the existence of a locally finite measure $\mu_1 \in \msc M(\Omega \times \mathbb R)$, with $\mr{supp}\, \mu_1 \subset \Omega \times [\inf u, \sup u]$, such that 
\begin{equation}\label{eq:Trho}
\langle T, \phi \varrho \rangle = L_\phi(\varrho)= \int_{\Omega} \int_{\mathbb R} \phi(t,x) \varrho(v) \dif \mu_1(t,x,v) \qquad \forall \phi \in C^\infty_c(\Omega), \quad \varrho \in C^\infty_c(\mathbb R).
\end{equation}
Finally, we calculate, for any $\phi \in C^\infty_c(\Omega)$ and $\eta^\prime \in C^\infty_c(\mathbb R)$, 
$$
\begin{aligned}
    \Big\langle \partial_t \chi + f^\prime(v) \nabla_x \chi, \, \phi \eta^\prime\Big\rangle & = \Big\langle \partial_t \chi + f^\prime(v) \nabla_x \chi, \, \phi (\eta^\prime-\eta^\prime(0)) \Big\rangle \\
    & +  \Big\langle \partial_t \chi + f^\prime(v) \nabla_x \chi, \, \phi \eta^\prime(0) \Big\rangle \\
    & = \Big\langle \partial_t \chi + f^\prime(v) \nabla_x \chi, \, \phi \int_0^v \eta^{\second} \Big\rangle \\
    & +  \Big\langle \partial_t \chi + f^\prime(v) \nabla_x \chi, \, \phi \eta^\prime(0) \Big\rangle \\
    & = -\Big\langle T, \phi \eta^\second \Big\rangle  +  \Big\langle \partial_t \chi + f^\prime(v) \nabla_x \chi, \, \phi \eta^\prime(0) \Big\rangle \\
    & = -\int_{\Omega} \int_{\mathbb R} \phi(t,x) \eta^{\second}(v) \dif \mu_1 + \eta^\prime(0) \int_{\Omega} \phi(t,x) \dif \mu_{\eta_0}
\end{aligned}
$$
where we recall that $\eta_0(u) = u$ is the trivial entropy and we used \eqref{eq:Trho} with $\varrho = \eta^{\second}$.
The result follows by setting $\mu_0 =\mu_{\eta_0} \times \delta_0$.
 \end{proof}

\section{Lagrangian Representation}\label{sec:lagrangian}
In this section we introduce the concept of \textit{Lagrangian representation} of quasi-entropy solution to the balance law \eqref{eq:conslaw}, which is the main tool that we will use throughout the rest of this chapter. This kind of representation was introduced for the first time in \cite{BBM_multid} for entropy solutions to scalar conservation laws:
\begin{equation}\label{eq:scld}
\partial_t u + \mr{div}_x \, f(u) \, = \, 0.
\end{equation}

In \cite{BBM_multid}, the authors use the Lagrangian representation to prove that continuous solutions to the conservation law \eqref{eq:scld} do not dissipate any entropy.

In \cite{marconi_structure} the Lagrangian representation is introduced also for finite entropy solutions to the conservation law \eqref{eq:scld} (i.e., no source terms are present, but the measures $\mu_{\eta}$ in \eqref{eq:entropy}  can change sign) and later on, in \cite{marconi_burgers}, it is used to prove that, in the one dimensional case and for Burgers equation $f(u) = u^2/2$, the entropy dissipation measures are concentrated on a $1$-rectifiable set. In this case the existence of a Lagrangian representation was proved via an explicit approximation procedure, similar to the transport-collapse scheme. In our setting, inspired by the method in \cite{Bon17}, where it is used for the continuity equation, we adopt a different approach based on Smirnov's Theorem on the representation of $1$-currents. This method provides finer control over the behavior of the Lagrangian curves, which will be essential in the subsequent analysis.

\subsection{Definition of the Lagrangian Representation}
In this subsection we define the Lagrangian representation of a quasi-entropy solution to the balance law \eqref{eq:conslaw}.
In the following, we denote by $\Gamma$ the space of curves
$$
\Gamma \doteq   \Big\{(I_{\gamma}, \gamma)\; \big| \;  \gamma = (\gamma^x, \gamma^v):I_{\gamma} \to \mathbb R^d_x \times \mathbb R^+_v,\; \; \gamma^x\;  \text{is Lipschitz and $\gamma^v$ is in $BV(I_\gamma)$}\Big\}
$$
where $I_{\gamma} = (t^1_{\gamma}, t^2_{\gamma})$ is the domain of definition of a curve $\gamma$.
A locally finite measure on $\Gamma$ is a measure $\bs \omega \in \mathscr M(\Gamma)$ such that $\bs \omega(K) < +\infty$ for every set $K \subset \Gamma$ of the form 
$$
 K = \Big\{  \gamma \in \Gamma \; \big| \;|\gamma^x(0)|, |\gamma^v(0)| \leq M, \quad \mathrm{Lip} (\gamma^x) \leq M, \quad \| \gamma^v\|_{BV} \leq M \Big\}.
$$

\begin{defi}\label{def:lagrangianrep}
A \textit{Lagrangian representation} (of the hypograph) of a quasi-entropy solution $u$ to \eqref{eq:conslaw} is a locally finite measure $\bs \omega \in \msc {M}^+(\Gamma)$ such that:
\begin{enumerate}
    \item it holds
    \begin{equation}\label{eq:deflagromega}
        \chi \cdot \mathscr L^{d+2} = \int_{\Gamma} (\mathrm{id}, \gamma)_{\sharp} \mathscr L^1 \llcorner I_{\gamma} \dif \bs \omega(\gamma)
    \end{equation}
    where $\chi$ is defined as in \eqref{eq:chidef}, and the measure in the right hand side is defined on test functions $\varphi \in  C_c(\mathbb R^{d+2})$ by
    $$
    \int_{\mathbb R^{d+2}} \varphi \dif \Big[ \int_{\Gamma} (\mathrm{id}, \gamma)_{\sharp} \mathscr L^1 \llcorner I_{\gamma} \dif \bs \omega(\gamma)\Big] := \int_\Gamma \Big( \int_{I_\gamma} \varphi(t,\gamma(t)) \dif t\Big) \dif \bs \omega(\gamma).
    $$
    \item $\omega$ is concentrated on the set of curves $\gamma = (\gamma^x, \gamma^v) \in \Gamma$ such that 
    \begin{equation}\label{eq:charaeqla}
    \begin{aligned}
   &  \dot{\gamma}^x = f^{\prime}(\gamma^v(t)), \qquad \text{for $\msc L^1$-a.e. $t \in I_{\gamma}$}, 
   \end{aligned}
    \end{equation}    
    \item the following integral bound holds: for every compact $K \subset \mathbb R^d$ and $T > 0$, one has
    \begin{equation}\label{eq:CKTlr}
        \begin{aligned}
     \int_{\Gamma}  \mathrm{Tot.Var.}\big(\gamma^v, \, I_\gamma \cap [0,T]\big) \cdot \mathbf 1_{\{\gamma^x(t_\gamma^1) \in K\}}(\gamma) \,\dif \bs \omega(\gamma) < C_{K,T}\\
\sum_{i=1}^2 \bs \omega \Big(\big\{ \gamma \in \Gamma \; | \; (t,\gamma^x(t^i_\gamma)) \in (0, T) \times K \big\} \Big) \leq \wt C_{K,T}
     \end{aligned}
    \end{equation}
\end{enumerate}
\end{defi}

The link between the notion of kinetic solution and Lagrangian representation is given by the following proposition.
\begin{thm}\label{thm:lagrkin}
A function $u \in \mathbf L^\infty(\Omega, \; \mathbb R)$ is a quasi-entropy solution of \eqref{eq:conslaw}  if and only if it admits a Lagrangian representation.
\end{thm}
We prove that if $u$ admits a Lagrangian representation, then $u$ is a quasi-entropy solution. The opposite implication is deeper: the proof, which relies on the structure of 1-dimensional normal currents in $\mathbb R^n$, is carried out in the next subsection (see Theorem \ref{thm:lagraex}). Actually, for the opposite implication, we will prove a stronger result, stating the existence of a \emph{good} Lagrangian representation (see Definition \ref{defi:goodlr})
\begin{proof}
Assume that $u$ admits a Lagrangian representation $\bs \omega$. We have to show the existence of a pair of locally finite measures $\mu_0, \mu_1 \in \msc M(\Omega \times \mathbb R)$ such that \eqref{eq:kinu} holds. To do so, we consider the distribution $$S \doteq \partial_t \chi + f^{\prime}(v) \cdot \nabla_x \chi$$ and calculate, for every test functions $\eta^\prime \in \msc D(\mathbb R)$ and $\phi \in \msc D(\Omega)$,
\begin{equation}\label{eq:Scomput}
    \begin{aligned}
-\langle S, \eta^{\prime} \phi \rangle & \; \; \; =   \int_\Omega \int_\mathbb R \eta^{\prime}(v) \chi(t, x, v) [\phi_t(t, x) + f^{\prime} (v)\cdot \nabla_x \phi (t, x)] \dif v \dif x \dif t\\
          &  \stackrel{\text{by \eqref{eq:deflagromega}}} {=}   \int_{\Gamma}  \int_{I_{\gamma}} \eta^{\prime}(\gamma^v(t)) [\phi_t(t, \gamma^x(t)) + f^{\prime} (\gamma^v(t))\cdot \nabla_x \phi (t, \gamma^x(t))] \dif t \dif \bs \omega(\gamma) \\
          & \stackrel{\text{by \eqref{eq:charaeqla}}}{=} \int_{\Gamma} \int_{I_{\gamma}} \eta^{\prime}(\gamma^v(t)) [\phi_t(t, \gamma^x(t)) + \dot \gamma^x(t)\cdot \nabla_x \phi (t, \gamma^x(t))] \dif t \dif \bs \omega(\gamma) \\
          & \; \;  \;=    \int_{\Gamma} \int_{I_{\gamma}} \eta^{\prime}(\gamma^v(t)) \frac{\dif}{\dif t} \phi(t, \gamma^x(t)) \dif t\dif \bs \omega(\gamma).
    \end{aligned}
\end{equation}
Now we compute more explicitly the integrand in the last line: in particular, since $\eta^\prime \phi$ has compact support, by (3) we deduce that for almost every $\gamma$ the function  $t \mapsto g_{\gamma}(t)\doteq \eta^{\prime}(\gamma^v(t)) \phi(t, \gamma^x(t))$ is in $BV(I_\gamma)$. Let $J_\gamma$ be the jump set of $\gamma^v$, defined by 
$$
J_\gamma \doteq \big\{ t \in I_\gamma \; | \; \gamma^v(t+) \neq \gamma^v(t-)\big\}.
$$
By Volpert's chain rule we obtain that $ D g_{\gamma} \in \msc M(I_{\gamma})$ is given by 
$$
\begin{aligned}
     D g_{\gamma}   =   \eta^{\prime}(\gamma^v(t))& \,\frac{\dif}{\dif t}\phi(t, \gamma^x(t)) \cdot \msc L^1 \,+\, \eta^{\second}(\gamma^v(t)) \phi(t, \gamma^x(t)) \cdot \wt D \gamma^v(t) \\
    & +  \sum_{t \in J_\gamma} \big(\eta^{\prime}(\gamma^v(t+))-\eta^{\prime}(\gamma^v(t-))\big) \cdot \phi(t, \gamma^x(t)) \cdot \delta_t
\end{aligned}
$$
where $\wt D\gamma^v$ denotes the diffuse part of the derivative of $\gamma^v$, see \cite[\S 3.9]{AFP_book}.
Therefore, using the chain rule in \eqref{eq:Scomput}, we obtain 
\begin{equation}\label{eq:integrandcompS}
    \begin{aligned}
    \int_{I_{\gamma}} \eta^{\prime}(\gamma^v(t)) \frac{\dif}{\dif t} \phi(t, \gamma^x(t)) \dif t & = \int_{I_{\gamma}} \dif D g_{\gamma}(t) \\
  & -\int_{I_{\gamma}} \eta^{\second}(\gamma^v(t)) \phi(t, \gamma^x(t)) \dif \wt D \gamma^v(t)\\ 
  & - \sum_{t\in J_\gamma} (\eta^{\prime}(\gamma^v(t+))-\eta^{\prime}(\gamma^v(t-))) \cdot \phi(t, \gamma^x(t)).
    \end{aligned}
\end{equation}
For every $\gamma \in \Gamma$, define the measure $\mu_1^{\gamma} \in \msc M(\Omega \times \mathbb R)$ as 
\begin{equation}\label{eq:mu1def}
    \begin{aligned}
    \mu_1^{\gamma}& \doteq -(\mathbb I , \gamma)_{\sharp} \Tilde{D} \gamma^v \\
    & - \mathscr H^1 \llcorner \{(t,x,v) \; | \; \gamma^v(t-) \leq v \leq \gamma^v(t+), \quad t \in J_\gamma\}\\
    & +\mathscr H^1 \llcorner \{(t,x,v) \; | \; \gamma^v(t+) \leq v \leq \gamma^v(t-), \quad t \in J_\gamma\},
    \end{aligned}
 \end{equation}
and the measure $\mu_0^{\gamma}\in \msc M(\Omega \times \mathbb (0,1))$ as
\begin{equation}\label{eq:mu0def}
 \mu_0^{\gamma} \doteq \delta_{(t^1_{\gamma}, \gamma(t^1_{\gamma}))} - \delta_{(t^2_{\gamma}, \gamma(t^2_{\gamma}))}.
\end{equation}
Using that 
$$
\int_{I_\gamma} \dif D g_\gamma(t) = \eta^\prime(\gamma(t_\gamma^2)) \phi(t, \gamma^x(t_\gamma^2))-  \eta^\prime(\gamma(t_\gamma^1)) \phi(t, \gamma^x(t_\gamma^1))
$$
and inserting \eqref{eq:integrandcompS} in the last line of \eqref{eq:Scomput}, we obtain, 
\begin{equation}
    \begin{aligned}
   & -\big\langle S, \eta^{\prime} \phi\big\rangle & =  \int_{\Gamma} \big\langle -\mu^{\gamma}_0, \; \eta^{\prime} \phi\big\rangle \dif \bs \omega(\gamma)  + \int_{\Gamma} \big\langle \mu^{\gamma}_1, \; \eta^{\second} \phi\big\rangle \dif \bs \omega(\gamma).
    \end{aligned}
\end{equation}
This shows that setting 
\begin{equation}\label{eq:inducedpair}
\begin{aligned}
    & \mu_1 : = \int_{\Gamma} \mu_1^{\gamma}\dif \bs \omega(\gamma), \qquad \mu_0 := \int_{\Gamma} \mu_0^{\gamma} \dif \bs \omega(\gamma)
    \end{aligned}
\end{equation}
it holds
\begin{equation}
    S = \partial_v \mu_1 + \mu_0, \qquad \text{in} \; \msc D^{\prime} (\Omega \times \mathbb R)
\end{equation}
    which proves the claim. Finally, notice that $\mu_1, \mu_0$ are locally finite measures, by assumption (3) in Definition \ref{def:lagrangianrep}. Therefore $u$ is a quasi-entropy solution by Proposition \ref{prop:feskin}.
\end{proof}

We say that a pair $(\mu_0, \mu_1) \in \mathscr M(\Omega \times \mathbb R)^2$ is induced by a Lagrangian representation if \eqref{eq:inducedpair} holds, with $\mu_1^{\gamma}, \mu_0^{\gamma}$ defined in \eqref{eq:mu1def}, \eqref{eq:mu0def}. Every Lagrangian representation induces a unique pair $(\mu_0, \mu_1)$.
We give the following definition.
\begin{defi}\label{defi:goodlr}
We say that a Lagrangian representation $\bs \omega$ is \emph{good} if the pair $(\mu_0, \mu_1) \in \mathscr M(\Omega \times \mathbb R)^2$ induced by $\bs \omega$ satisfies additionally
\begin{equation}\label{eq:optcond}
    \begin{aligned}
        \left| \mu_1 \right| = \int_{\Gamma} |\mu_1^{\gamma}| \dif \bs \omega(\gamma), \qquad \left| \mu_0 \right| = \int_{\Gamma} |\mu_0^{\gamma}| \dif \bs \omega(\gamma).
    \end{aligned}
\end{equation}
\end{defi}

Condition \eqref{eq:optcond}  says that there are no cancellations, e.g. the second condition in \eqref{eq:optcond} implies that no curves are created in the same point where others are canceled, that is, up to a negligible set of curves, for every pair of curves $\gamma, \sigma$ there holds
$$
(t_\gamma^1, \gamma(t_\gamma^1)) \neq (t_\sigma^1, \sigma(t_\sigma^1)) .
$$
The analogous condition for $\mu_1^\gamma$ means that, up to a negligible set of curves, there holds
$$
\{(t,x,v) \;  | \; \gamma^x(t) = \sigma^x(t) = x, \quad \gamma^v(t+) < v < \gamma^v(t-), \quad \sigma^v(t+) > v > \sigma^v(t-)\} = \emptyset
$$
$$
\{(t,x,v) \;  | \; \gamma^x(t) = \sigma^x(t) = x, \quad \gamma^v(t+) > v > \gamma^v(t-), \quad \sigma^v(t+) < v < \sigma^v(t-)\} = \emptyset
$$

\begin{remark}
    Not every couple $(\mu_0, \mu_1) \in \mathscr M(\Omega \times \mathbb R)^2$ is induced by a good Lagrangian representation. For example, every couple $(\mu_0, \mu_1) \neq 0 $  which corresponds to $u = $ const., i.e. $\mu_0 + \partial_v \mu_1 = 0$,  is not induced by any good Lagrangian representation.
\end{remark}

\subsection{Existence of (good) Lagrangian Representations.}
In this subsection we use a Theorem of Smirnov \cite{smirnov_currents}  about the decomposition of  1-currents into currents of the form $\llbracket\gamma\rrbracket$ associated to Lipschitz curves (see Example \ref{ex:curvents}) to derive good  Lagrangian representations for \textit{all} quasi-entropy solutions. As a byproduct, we obtain a proof of the remaining implication of Proposition \ref{thm:lagrkin}.

\subsubsection{Preliminaries About the Theory of Currents} For an introduction to the subject we refer for example to \cite{KP_book}.
The space $\mathscr D_k(\mathbb R^d)$ of \textit{$k$-dimensional currents} is the dual of the space $\mathscr D^k(\mathbb R^d)$ of all smooth $k$-differential forms with compact support in $\mathbb R^d$. 

Given a $k$-current $\ms T \in \mathscr D_k(\mathbb R^d)$, its \textit{mass} $\mathbb M(\ms T)$ is defined as 
$$
\mathbb M (\ms T) = \sup\{\langle \ms T \; , \omega\rangle \; | \; \omega \in \mathscr D^k(\mathbb R^d), \quad |\omega | \leq 1\}
$$
\begin{remark}[1-currents with finite mass]
The space of 1-currents with finite mass can be identified with Radon vector measures $\ms T \in \mathscr M(\mathbb R^d)^d$, and we will often identify a 1-current with finite mass with the associated vector measure. The duality of a 1-current  with finite mass $\ms T = (\ms T_1, \ldots, \ms T_d)$ with a vector field $\Phi = (\Phi^1, \ldots, \Phi^d)\in C^\infty_c(\mathbb R^d,\;  \mathbb R^d)$ is 
$$
\langle \ms T \; , \; \Phi\rangle =  \sum_{i=1}^n \int_{\mathbb R^d}\Phi^i \dif \ms T_i.
$$
Analogously, the space of $0$-currents with finite mass can be identified with the space of Radon measures on $\mathbb R^d$.
\end{remark}

\begin{exmp}[Rectifiable $k$-currents]
Let  a $k$-rectifiable set $M \subset \mathbb R^d$ oriented by a unit $k$-vector $\tau$, and a measurable  multiplicity function $\theta$ be given. Then we denote by $\llbracket M ,\; \tau, \; \theta\rrbracket$ the $k$-current defined by 
$$
\langle \llbracket M ,\; \tau, \; \theta\rrbracket, \; \omega \rangle = \int_M \theta(x)\langle \tau(x), \; \omega(x) \rangle \dif \msc H^k(x), \qquad \forall \; \omega \in \msc D^k(\mathbb R^d)
$$
\end{exmp}
The \textit{boundary} of a k-current $\ms T$ is a $k-1$-current $\partial \ms T \in \mathscr D_{k-1}(\mathbb R^d)$ defined as
$$
\langle \partial \ms T \; , \; \omega\rangle = \langle \ms T \; , \; d \omega\rangle, \qquad \forall \; \omega \in \mathscr D^{k-1}(\mathbb R^d)
$$
where $d : \mathscr D^{k-1}(\mathbb R^d) \to \mathscr D^{k}(\mathbb R^d)$ is the De Rham's-Cartan differential.

\begin{defi}[Normal Currents]
We say that a $k$-current $\ms T \in \mathscr D_{k}(\mathbb R^d)$ is \textit{normal} if both $\ms T$ and $\partial \ms T$ are of finite mass. 
\end{defi}

\begin{defi}[A-cyclic currents]
    Let $\ms T \in  \msc D_k(\mathbb R^d)$ be a $k$-current. We say that 
\begin{enumerate}
    \item A current $\ms C \in \msc D_k(\mathbb R^d)$ is a \textit{subcurrent} of $\ms T$, and we write $\ms C \leq \ms T$, if $\mathbb M (\ms T) = \mathbb M (\ms T-\ms C) +\mathbb M (\ms C)$.
    \item A current $\ms C \leq \ms T$ is a \textit{cycle} of $ \ms T$ if $\partial \ms C = 0$.
    \item $\ms T$ is a-cyclic if its only cycle is $\ms C = 0$.
\end{enumerate}
\end{defi}

\begin{remark}
    In general, given $\ms T, \ms C \in \msc D_k(\mathbb R^d)$, it is clear by the definition of mass of a current that it holds $\mathbb M (\ms T) \leq \mathbb M (\ms T-\ms C) +\mathbb M (\ms C)$. Therefore $\ms C$ is a subcurrent of $\ms T$ if and only if $\mathbb M (\ms T) \geq \mathbb M (\ms T-\ms C) +\mathbb M (\ms C)$. In particular, if $\ms T$ is of finite mass, $\ms C$  must be of finite mass.
\end{remark}

\begin{exmp}[Structure of 1-subcurrents]\label{rem:substru}
    Let $\ms T \in \msc D_1(\mathbb R^d)$ be a 1-current with finite mass, and let $\left\Vert \ms T\right\Vert \in \msc M^+(\mathbb R^d)$ be its total variation measure. By the polar decomposition of vector valued measures (see \cite{AFP_book}) there exists a unit measurable vector field $\vec S: \mathbb R^d  \to \mathbb R^d$, defined $\left\Vert \ms T\right\Vert$-a.e., such that $\ms T = \vec S \cdot \left\Vert \ms T \right\Vert$. It is easy to see that every subcurrent $\ms C \in \msc D_1(\mathbb R^d)$  of $\ms T$ has the form 
    $$
    \ms C = \vec S \cdot  \left\Vert \ms C \right\Vert, \qquad  \left\Vert \ms C \right\Vert \leq  \left\Vert \ms T \right\Vert \quad \text{as measures}
    $$
\end{exmp}

\begin{exmp}[$1$-currents associated with Lipschitz curves]\label{ex:curvents}
The  most important example for the following of this section is the case of  1-rectifiable currents $$\llbracket \gamma \rrbracket : = \llbracket \gamma\big((0,1)\big) ,\; \tau(x), \; \msc H^0\big(\gamma^{-1}(\{x\})\big) \rrbracket \in \msc D_1(\mathbb R^d)$$ associated with a Lipschitz curve $\gamma: [0,1] \to \mathbb R^d$, where $\tau(x)$ is the tangent vector the the graph $\gamma((0,1))$ at a point $x \in \gamma((0,1))$, defined for $\mathscr H^1$-almost every point $x \in \gamma((0, 1))$, and
 $\msc H^0\big(\gamma^{-1}(\{x\})\big)$ counts how many times the curve $\gamma$ passes through the point $x$.  The action of $\llbracket \gamma \rrbracket$ on 1-forms (which can be identified with vector fields $\Phi \in C^\infty_c(\mathbb R^d, \mathbb R^d)$) can be explicitly written as (see the Area Formula in \cite{AFP_book})  
$$
\begin{aligned}
\langle \llbracket \gamma \rrbracket, \; \Phi \rangle & = \int_{\gamma((0, 1))}\msc H^0\big(\gamma^{-1}(\{x\})\big)\cdot  \Big\langle  \frac{\dot{\gamma}}{|\dot{\gamma}|}, \; \Phi\Big\rangle \cdot \dif \msc H^1 \\
& = \int_0^1 \langle \Phi(\gamma(\tau)), \;  \dot \gamma(\tau)\rangle  \dif \tau, \qquad \forall \; \Phi \in  C^\infty_c(\mathbb R^d, \mathbb R^d)
\end{aligned}
$$
In particular $\llbracket \gamma \rrbracket$ has finite mass since 
$$
\langle \llbracket \gamma \rrbracket, \; \Phi \rangle  \leq \int_0^1 |\dot{\gamma}(\tau)| \dif \tau  \leq \mr{Lip} (\gamma), \qquad \forall \; \Phi \in \mc C_c^{\infty}(\mathbb R^d, \mathbb R^d), \quad |\Phi| \leq 1.
$$
The boundary of the 1-current $\llbracket \gamma \rrbracket \in \msc D_1(\mathbb R^d)$, satisfies
$$
\langle \partial \llbracket \gamma \rrbracket, \; \phi \rangle  = \int_0^1 \frac{\dif}{\dif \tau} \phi(\gamma(\tau)) \dif \tau = \phi(\gamma(1))-\phi(\gamma(0)), \qquad \forall \; u \in C^\infty_c(\mathbb R^d)
$$
so that $\partial \llbracket \gamma \rrbracket = \delta_{\gamma(1)}- \delta_{\gamma(0)} \in \msc M(\mathbb R^d)$. Therefore $\llbracket \gamma \rrbracket$ is a normal 1-current. Moreover, a 1-current $\llbracket \gamma \rrbracket \in \msc D_1(\mathbb R^d)$ is a cycle if and only if $\gamma(1) = \gamma(0)$, otherwise it is an a-cyclic 1-current.
\end{exmp}

\subsubsection{Smirnov's Theorem for Normal Currents}
The following Theorem of Smirnov yields a decomposition of $1$-normal currents into superposition of currents associated to Lipschitz curves.
\begin{thm}[Smirnov, \cite{smirnov_currents}]\label{thm:smirnov}
Let $U \subset \mathbb R^d$ be an open set. Let $\ms T \in \mathscr D_1(U)$ be a normal, a-cyclic 1-current in $U$. Then there exists a positive measure $\boldsymbol{\eta} \in \mathscr M^+(\mathrm{Lip}((0,1)  ;  U))$ such that 
\begin{equation}\label{eq:Tdecomp}
    \langle \ms T, \omega \rangle = \int_{\mathrm{Lip}((0,1)  ;  U)} \langle \llbracket\gamma\rrbracket, \omega\rangle \dif \boldsymbol{\eta} (\gamma), \qquad \forall \; \omega \in \mathscr D^1(U)
\end{equation}
Furthermore, the mass of $\ms T$ decomposes as 
\begin{equation}\label{eq:Mdecomp}
    \mathbb M (\ms T)  =  \int_{\mathrm{Lip}((0,1)  ; U)}  \mathbb M (\llbracket\gamma\rrbracket) \dif \boldsymbol{\eta} (\gamma)
\end{equation}
and the boundary measure $\partial T$ decomposes as 
\begin{equation}\label{eq:dTdecomp}
    (\partial T)^+ = (e_1)_{\sharp} \boldsymbol{\eta}, \qquad  (\partial T)^- = (e_0)_{\sharp} \boldsymbol{\eta}
\end{equation}
where 
$$
e_t : \mathrm{Lip}((0,1)  ;  U) \to \mathbb R^d, \qquad e_t(\gamma) = \gamma(t)
$$
is the evaluation map at a point $t \in (0, 1)$ 
and $(\partial T)^{\pm}$ denotes the positive/negative part of the measure $\partial T$.
\end{thm}

\begin{remark}\label{eq:intcurv}
    Let $\ms T = \vec S \cdot \left\Vert \ms T\right\Vert$ be the polar decomposition of the vector measure $\ms T$. From the formula of decomposition of the mass it follows that the curves on which $\boldsymbol{\eta}$ is concentrated are such that $\vec S(\gamma(\tau))$ is parallel to $\dot \gamma(\tau)$ for $\msc L^1$ a.e. $\tau$. In fact, since $\ms T$ is of finite mass, we can use $\vec S$ as a test vector field in equation \eqref{eq:Tdecomp} (see for example \cite{KP_book}). This yields
    \begin{equation}\label{eq:paral1}
    \begin{aligned}
     \mathbb M(\ms T)  = \langle \ms T, \; \vec S \rangle & = \int_{\mathrm{Lip}((0,1)  ; U)} \langle \llbracket\gamma\rrbracket, \vec S \rangle \dif \boldsymbol{\eta} (\gamma) = \int_{\mathrm{Lip}((0,1)  ; U)} \int_0^1 \langle \dot \gamma(\tau), \; \vec S(\gamma(\tau)) \rangle \dif \tau \dif \boldsymbol{\eta} (\gamma)
    \end{aligned}
    \end{equation}
    But on the other hand, by equation \eqref{eq:Mdecomp}, it also holds
    \begin{equation}\label{eq:paral2}
    \mathbb M(\ms T) = \int_{\mathrm{Lip}((0,1)  ; U)}  \mathbb M (\llbracket\gamma\rrbracket) \dif \boldsymbol{\eta} (\gamma) = \int_{\mathrm{Lip}((0,1)  ; U)} \int_0^1 |\dot \gamma(\tau)| \dif \tau \dif \boldsymbol{\eta} (\gamma)
    \end{equation}
   Since $\vec S$ is a unitary vector, it holds
   $$
   \langle \dot \gamma(\tau), \; \vec S(\gamma(\tau)) \rangle \leq |\dot \gamma(\tau)| , \qquad \text{for $\msc L^1$ a.e. $\tau \in (0, 1)$}
   $$
    Therefore combining \eqref{eq:paral1}, \eqref{eq:paral2} one obtains that $\dot{\gamma}(\tau)$ must be parallel to $\vec S$ for $\mc L^1$ a.e. $\tau \in (0,1)$, for $\bs \eta$ a.e. $\gamma \in \Gamma$.
\end{remark}

We fix some notation that we will need for the following proposition. We let $\Gamma$ be as above, and as usual elements of $\Gamma$ will be denoted by $\gamma = (\gamma^x, \gamma^v)$. Instead, we let $$\Tilde{\Gamma} = \mathrm{Lip}((0,1),\, \Omega \times \mathbb R).$$ Elements of $\Tilde{\Gamma}$ are denoted by $$\Theta(\tau) \doteq \big(\Theta^t(\tau), \Theta^x(\tau), \Theta^v(\tau)\big) \in \Omega \times \mathbb R, \quad\tau \in (0, 1).$$

\begin{thm}\label{thm:lagraex}
    Let $u$ be a quasi-entropy solution and let $(\mu_0, \mu_1) \in \mathscr M(\Omega \times \mathbb R)^2$ be as in Proposition \ref{prop:feskin}. Then there exists a good Lagrangian representation $\boldsymbol{\omega} \in \mathscr M^+(\Gamma)$ of $u$.  Moreover, letting $(\wt\mu_0, \wt \mu_1)$ be the pair induced by $\bs \omega$ as in Definition \ref{defi:goodlr}, it holds $|\wt \mu_0| \leq |\mu_0|$ and
    \begin{equation}\label{eq:newpairest}
        \wt \mu_1^+ \leq \mu_1^+, \qquad \wt \mu_1^- \leq \mu_1^-, \qquad \text{as measures}.
    \end{equation}
\end{thm}

\begin{proof}
\textbf{1.}   Define the 1-current $\ms T \in \mathscr D_1(\Omega \times \mathbb R)$ as
  \begin{equation}\label{eq:Tdef}
  \begin{aligned}
    \ms T = \Big( \chi \cdot \msc L^{d+2}, \quad  \chi \cdot f^{\prime}(v) \llcorner \msc L^{d+2},\quad -\mu_1\Big)^T  \in \mathscr D_1(\Omega \times \mathbb R).
      \end{aligned}
    \end{equation}
    In this step we prove that $\ms T$ is a normal, a-cyclic current.
When testing $\ms T$ against a smooth vector field, we obtain
\begin{equation}
    \begin{aligned}
        \langle \ms T \; , \; \Phi \rangle  & = \int_{\Omega \times \mathbb R} \chi \cdot [ \Phi^t + \Phi^x \cdot f^{\prime}(v) ] \dif t \dif x \dif v - \int_{\Omega \times \mathbb R} \Phi^v \dif \mu_1   \\
        &  \leq \sqrt{1 + |f^{\prime}(v)^2|}\cdot  \int_{\Omega \times \mathbb R} \chi \dif t \dif x \dif v  +\Vert\mu_1\Vert_{\msc M}, \\
        & \qquad \forall\;  \Phi \in C^\infty_c(\Omega \times \mathbb R, \mathbb R^{d+2}), \qquad |\Phi| \leq 1.
    \end{aligned}
\end{equation}
The boundary of $\ms T$ is, by equation \eqref{eq:kinu}, the measure $\mu_0$, plus the boundary terms at $t = 0$ and $t = T$. In fact it holds
\begin{equation}
\begin{aligned}
\langle \partial \ms T \; , \; \phi\rangle  =&  \int_{\Omega \times \mathbb R} \chi \cdot [\partial_t \phi + \nabla  \phi_x \cdot f^{\prime}(v) ] \dif t \dif x \dif v - \int_{\Omega \times \mathbb R} \partial_v  \phi \dif \mu_1   \\
     =& - \int_{\Omega \times \mathbb R}  \phi \dif \mu_0 + \int_{\mathbb R^{d+1}} \chi(T) \cdot  \phi(T) - \chi(0) \cdot  \phi(0) \dif x \dif v \\
     =& -\big\langle \phi, \mu_0 \big\rangle + \Big\langle \delta_T\otimes \chi(t)\cdot \msc L^{d+1} -\delta_0\otimes \chi(0)\cdot \msc L^{d+1} \; , \; \phi \Big\rangle, \\
    & \qquad \forall \; \phi \in C^\infty_c(\Omega \times \mathbb R).
\end{aligned}
\end{equation}
Therefore $\ms T$ is a normal 1-current.

Finally, we  prove that $\ms T$ is a-cyclic. Let $\ms T = \vec {S} \cdot \left\Vert \ms T \right\Vert$ be the polar decomposition of the current (vector measure) $\ms T$. Moreover, we choose $A\subset \Omega \times \mathbb R$ measurable such that $\mu_1 = \mu_1^a + \mu_1^s =  \mu_1 \llcorner A + \mu_1\llcorner A^c$ is the decomposition of $\mu_1$ with respect to the Lebesgue measure into absolutely continuous and singular part respectively. In  particular, we can choose $A$ such that $\msc L^{d+2}(A^c) = 0$.
We let $\rho: \Omega \times \mathbb R \to \mathbb R$ be the density of $\mu^a_1$, i.e. $\mu_1^a=  \rho\msc L^{d+2}$.
Then the unit vector $\vec S$ is given by 
$$
(\vec S^t, \vec S^x, \vec S^v) = \vec S= \begin{cases}
\Big(\sqrt{1+\left|f^{\prime}(v)\right|^2 + \rho^2}\Big)^{-1}(1, f^{\prime}(v), \rho(t, x, v))^T, & \text{in} \; A\\
\\
(0, 0, \sigma(t,x,v))^T, & \text{in} \; A^c
\end{cases}
$$
where $\mu^s_1 = \sigma \left\Vert \mu^s_1\right\Vert$, $\sigma \in \{-1, 1\}$ $\mu^s_1$-a.e., is the polar decomposition of $\mu^s_1$.

As a preliminary step we prove that for every subcurrent $\ms C \leq \ms T$ such that $\ms C = \ms C_1 + \ms C_2$ with $0 \neq \ms C_1 \leq \ms T\llcorner A$ and $\ms C_2 \leq \ms T\llcorner A^c$, then $\partial \ms C \neq 0$. Let $R > 0$ be such that  $\mr{supp}\, \ms T \subset [0, T] \times B_R(0)\subset \Omega \times \mathbb R$.  Choose a test function $\phi \in \mathbf
 C^{\infty}_c(\Omega \times \mathbb R)$ such that
$$
\begin{cases}
\phi(t, x, v) = t, &  \text{if} \; (t, x, v) \in [0,T] \times B_R(0), \\
\phi(t, x, v) = 0, & \text{if}\; (t, x, v) \in  [0,T] \times B_{2R}(0)^c
\end{cases}
$$
Then it holds
$$
\langle \ms C \; , \; \nabla \phi \rangle = \int_{\Omega \times \mathbb R} \vec S^t\dif \left\Vert \ms C_1\right\Vert  >0
$$
because $\vec S^t(x) > 0$ for$\left\Vert \ms C_1\right\Vert$ a.e. $x \in \Omega \times \mathbb R$. To conclude that $\ms T$ is a-cyclic, we need to prove that for every $0 \neq \ms C \leq \ms T\llcorner A^c$, it holds $\partial \ms C \neq 0$. In this case, one has $\ms C = (0, 0, -\hat \mu_1)$, with $\hat \mu_1  \leq \mu_1$ (as $0$-currents), and $\partial \ms C = 0$ corresponds to $\partial_v \hat \mu_1 = 0$ in the sense of distributions. But since $\hat \mu_1$ is compactly supported in $v$, this implies $\hat \mu_1 = 0$, and therefore $\ms C = 0$, as wanted.

\vspace{0.3cm}
 \textbf{2.}  Thanks to Step 1, we can apply Smirnov Theorem \ref{thm:smirnov}, which gives a measure $\boldsymbol{\eta}\in \mathscr M^+(\wt \Gamma)$ such that \eqref{eq:Tdecomp} holds. Expanding the relation, this means that for every vector field $\Phi = (\Phi^t, \Phi^x, \Phi^v) \in C^\infty_c(\Omega \times \mathbb R)$, we have
\begin{equation}
    \begin{aligned}
   & \int_{\Omega \times \mathbb R}\chi \cdot [ \Phi^t + (f^{\prime}(v)\cdot \Phi^x ) ] \dif \mc L^{d+2}  -\int_{\Omega \times \mathbb R} \Phi^{v} \dif \mu_1 \\
   & = \int_{\wt \Gamma} \int_0^1 \dot\Theta(\tau) \cdot \Phi(\Theta(\tau))\dif \tau \dif \boldsymbol{\eta}(\Theta).
    \end{aligned}
\end{equation}
By Remark \ref{eq:intcurv} the  measure $\boldsymbol{\eta} \in \mathscr M^+(\wt \Gamma)$ given by Smirnov Theorem is concentrated on curves $\Theta \in \Tilde{ \Gamma}$ such that $\dot\Theta(\tau)$ is parallel to $\vec S(\Theta(\tau))$ for a.e. $\tau \in (0, 1)$. In particular, this means either $\Theta$ travels vertically along $v$, or the velocity of the first component is positive.

\vspace{0.3cm}
\textbf{3.}  Given the measure $\boldsymbol{\eta} \in \mathscr M^+(\wt \Gamma)$ in Step 2., in order to get a Lagrangian representation in the sense of Definition \ref{def:lagrangianrep}, we want to eliminate a specific set of \say{bad} curves. In particular, define $\Delta \subset \wt \Gamma$  as the set of curves such that $\Theta^t((0,1))$ is a singleton. These are the curves whose support is contained  on an hyperplane $\{t = s\}$. The new measure will be defined by 
$$
\Tilde{\boldsymbol{\eta}}  =\boldsymbol{\eta} \llcorner (\Tilde{\Gamma} \setminus \Delta).
$$
Notice that in any case, by the shape of the vector field $\vec S$, the curves satisfy $\dot \Theta^t(\tau) \geq 0$ for a.e. $\tau \in (0,1)$ for $\Tilde{\boldsymbol{\eta}}$ almost every $\Theta$. 
For curves in $\Tilde{\Gamma} \setminus \Delta$ such that the first component $\Theta^t$ is non-decreasing, it is well defined a \textit{reparametrization map} $\mc R :\Tilde{\Gamma} \setminus \Delta \to \Gamma$ defined as
\begin{equation}
    \mc R(\Theta)(t) \doteq \big(\Theta^x(\tau(t)), \Theta^v(\tau(t))\big) , \qquad \tau(t) \doteq \inf \{ \tau  \in (0,1) \; : \; t < \Theta^t(\tau)\}.
\end{equation}
Here $\tau(t)$ is just the pseudo-inverse of the increasing Lipschitz function $\Theta^t(\tau)$. The domain of the curve $ \mc R (\Theta)$ will be $\mr{int}[\Theta^t((0, 1))]$ (the interior).   

Therefore we are now allowed to consider the pushforward measure $$\boldsymbol{\omega} \doteq \mc R_{\sharp} \Tilde{\boldsymbol{\eta}}\, \in \, \mathscr M^+(\Gamma)$$  which is a good candidate to be a Lagrangian representation in the sense of Definition \ref{def:lagrangianrep}.

\vspace{0.3cm}
\textbf{4.} Finally, we prove that $\bs \omega$ is a Lagrangian representation. The second condition in Definition \ref{def:lagrangianrep} is trivial. For the first condition, we need to verify that 
\begin{equation}
   \int_{\Gamma}  \big[(\mr{id}, \gamma)_{\sharp} \msc L^1 \llcorner I_{\gamma}\big] \dif  \mc R_{\sharp} \wt{\boldsymbol{\eta}}(\gamma)= \chi \cdot \msc L^{d+2}.
\end{equation}
Therefore consider any  test function $\phi \in C^\infty_c(\Omega \times \mathbb R)$ and calculate

\begin{equation}
    \begin{aligned}
\int_{\Gamma} \int_{\Omega \times \mathbb R} \phi(t, x, v)\dif \big[(\mr{id}, \gamma)_{\sharp} & \msc L^1 \llcorner I_{\gamma}\big](t, x, v) \dif  \mc R_{\sharp} \wt{\boldsymbol{\eta}}(\gamma)  = \int_{\Gamma} \int_{I_{\gamma}} \phi(t, \gamma(t)) \dif t \dif \mc R_{\sharp} \wt{\bs \eta}(\gamma) \\
& =   \int_{\Gamma} \int_{I_{\gamma}} \phi(\Theta(\tau(t))) \dif t \dif \wt{\bs \eta}(\Theta) \\
& =  \int_{\Gamma} \int_{0}^1 \phi(\Theta(\tau)) (\Theta^t)^{\prime}(\tau) \dif \tau  \dif \wt{\bs \eta}(\Theta) \\
& =  \int_{\Gamma} \int_{0}^1 \phi(\Theta(\tau)) (\Theta^t)^{\prime}(\tau) \dif \tau  \dif \bs \eta(\Theta),
    \end{aligned}
\end{equation}
where the last equality holds because 
$$
\int_{\Delta} \int_0^1 (\Theta^t)^{\prime}(\tau) \dif \tau \dif \bs \eta = \int_{\Delta} \int_0^1 (\Theta^t)^{\prime}(\tau)  \dif \tau \dif \wt{\bs \eta} = 0
$$
since $(\Theta^t)^{\prime}(\tau) = 0$ for every $\tau \in (0, 1)$ if $\Theta \in \Delta$, therefore we are allowed to substitute $\Tilde{\bs \eta}$ with $\bs \eta$.
Moreover, for $\bs \eta$-a.e. $\Theta \in \wt \Gamma$, by the last observation in Step 2.,  it holds 
$$
(\Theta^t)^{\prime}(\tau)\cdot \mathbf 1_{A^c}(\Theta(\tau)) = 0, \qquad \text{for $\msc L^1$-a.e. $\tau \in (0, 1)$}
$$
Therefore we obtain 
\begin{equation}
    \begin{aligned}
       \int_{\Gamma} \int_{0}^1 \phi(\Theta(\tau)) (\Theta^t)^{\prime}(\tau) \dif \tau  \dif \bs \eta(\Theta)  = \int_{\Gamma} \int_{0}^1 \phi(\Theta(\tau)) (\Theta^t)^{\prime}(\tau) \cdot \mathbf 1_{A} (\Theta(\tau)) \dif \tau  \dif \bs \eta(\Theta).
    \end{aligned}
\end{equation}
Define the vector field 
$$
\Phi(t, x, v) = (\phi(t, x, v), 0, 0 ) \cdot \mathbf 1_{A}(t, x, v) \in C^\infty_c([0,T] \times \mathbb R^{d+1}).
$$
Using $\Phi$ as a test function in the representation of Step 2, we obtain
\begin{equation}
    \begin{aligned}
\int_{\Gamma} \int_{0}^1 \phi(\Theta(\tau)) (\Theta^t)^{\prime}(\tau) \cdot \mathbf 1_{A} (\Theta(\tau)) \dif \tau  \dif \bs \eta(\Theta)  & = \langle \ms T, \Phi \rangle = \langle \chi \cdot\msc L^{d+2}, \phi \rangle
    \end{aligned}
\end{equation}
as wanted. 
\end{proof}

\subsection{Epigraph and Simultaneous Lagrangian Representations}
We conclude this section with the proof of some useful properties of Lagrangian representations.

\subsubsection{Lagrangian representation for the epigraph.} Given a quasi-entropy solution of the balance law \eqref{eq:conslaw}, we constructed a Lagrangian representation for the function $\chi$ in \eqref{eq:chidef}, which is the characteristic function of the hypograph of $u$. In an entirely similar way, we can give the definition of Lagrangian representation for the \textit{epigraph} of $u$ 

\begin{defi}\label{defi:lagraepi}
We say that $\bs \omega_e$ is a Lagrangian representation for the epigraph of $u$ if conditions (2), (3) of Definition \ref{def:lagrangianrep} hold, and condition (1) is replaced by:
\begin{equation}\label{eq:deflagromegae}
        \chi_e \cdot \mathscr L^{d+2} = \int_{\Gamma} (\mathrm{id}, \gamma)_{\sharp} \mathscr L^1 \llcorner I_{\gamma} \dif \bs \omega_e(\gamma)
    \end{equation}
where 
\begin{equation}
    \chi_e(t,x,v) \doteq \begin{cases}
        1 & \text{if} \; u(t, x) \leq v , \\
        0 & \text{otherwise}
    \end{cases}
\end{equation}
\end{defi}
The existence of such a representation follows from Theorem \ref{thm:lagraex}: in fact, define
$$
\Tilde{u}(t, x)\doteq 1-u(t, x), \qquad \bs g(u) \doteq -f(1-u).
$$
Then, if $\wt \chi(t, x, v) = \mathbf 1_{\mr{hyp} \; \wt u}(t, x, v)$, $\wt \chi$ is a solution of the kinetic equation 
$$
\partial_t \wt \chi + \bs g^{\prime}(v) \cdot \nabla_x \wt \chi = \partial_v \big[ \alpha_{\sharp} \mu_1\big] - \alpha_{\sharp}\mu_0, \qquad \text{in} \; \msc D^{\prime}_{t, x, v}
$$
where $\alpha: \Omega \times (0, 1) \mapsto \Omega \times (0, 1)$ is defined as $\alpha(t, x, v) = \alpha(t, x, 1-v)$.
 Since $\wt u$ is a quasi-entropy solution with flux $\bs g$ by Proposition \ref{prop:feskin}, Theorem \ref{thm:lagraex} yields a Lagrangian representation of $\wt u$, which we call $\wt{\bs \omega}$. Now, letting the map $T : \Gamma \to \Gamma$ be defined as
$$
T(\gamma)(t) =\alpha(\gamma(t)), \qquad t \in I_{\gamma}
$$
it is easy to see that, since 
$$
\wt \chi(t,x,v) =  \chi_e(t,x,1-v)
$$
the measure $$\bs \omega_e \doteq T_{\sharp} \wt{\bs \omega}$$ is a Lagrangian representation of the epigraph of $u$. 

\subsubsection{Simultaneous Lagrangian Representations}\label{subsec:simlagra}
Let $u$ be a given quasi-entropy solution. To keep the notation as clear as possible, when we need to consider simultaneously Lagrangian representations of the hypograph and of the epigraph, we will call them $\bs \omega_h$ and $\bs \omega_e$, respectively. Moreover, we will let
$$
\chi_h\doteq\mathbf 1_{\mr{hyp} \; u}, \qquad \chi_e\doteq \mathbf 1_{\mr{epi} \; u}.
$$
It is clear that, if $\chi_h$ satisfies \eqref{eq:kinu} for some $\mu_0, \mu_1$, then $\chi_e = 1-\chi_h$ satisfies \eqref{eq:kinu} with $(-\mu_0, -\mu_1)$. Motivated by this observation, we give the following definition.
\begin{defi}\label{defi:siminduc}
    We say that a pair $(\mu_0, \mu_1)$ is \textit{simultaneously induced} by Lagrangian representations $\bs \omega_h$, $\bs \omega_e$ of the hypograph and of the epigraph respectively if $(\mu_0, \mu_1)$ is induced by $\bs \omega_h$ and $(-\mu_0, -\mu_1)$ is induced by $\bs \omega_e$.
\end{defi}
It is useful to introduce the following relation $\preceq$ between pairs $(\mu_0, \mu_1)$ that belong to the set
\begin{equation}\label{eq:pairset}
\mathbf P_{\chi} := \big\{(\mu_0, \mu_1) \in \mathscr M(\Omega \times \mathbb R) \; | \; \text{\eqref{eq:kinu} holds}\big\}.
\end{equation}
\begin{defi}\label{defi:relation}
Given $(\mu_0, \mu_1)$ and $(\widehat \mu_0, \widehat \mu_1)$ in $\mathbf P_{\chi}$, we write 
\begin{equation}\label{eq:relation}
\begin{aligned}
    (\widehat \mu_0, \widehat \mu_1) \preceq (\mu_0,  \mu_1) \Longleftrightarrow 
        \begin{cases}
             \widehat \mu^+_1 \leq \mu^+_1, \qquad \widehat \mu^-_1 \leq \mu^-_1 \qquad \text{as measures} \\
             |\widehat \mu_0| \leq |\mu_0| 
                \end{cases}
    \end{aligned}
\end{equation}
\end{defi}
Note that $\preceq$ is the natural relation respected by the application of Theorem \ref{thm:lagraex}. This allows to prove the following proposition, which states the existence of simultaneously induced pairs.
\begin{prop}\label{prop:goodpairs}
    Let $u$ be a quasi-entropy solution. Every pair $(\widehat \mu_0, \widehat \mu_1) \in \mathbf P_\chi$  which is minimal with respect to $\preceq$ is simultaneously induced by some good Lagrangian representations $\bs \omega_h, \bs \omega_e$.
\end{prop}
\begin{proof}

Let $(\widehat \mu_0, \widehat \mu_1)$ be a minimal element  of $\mathbf P_{\chi}$ for the relation $\preceq$. It is easy to see there exists at least one minimal element: consider a starting pair $(\bar \mu_0, \bar \mu_1)$ and consider the set 
$$
\Sigma := \{(\mu_0, \mu_1) \; | (\mu_0, \mu_1) \preceq (\bar \mu_0, \bar \mu_1)\}.
$$
Consider now a minimizer for the function $\mc F: (\mu_0,  \mu_1) \mapsto\Vert \mu_1\Vert_{\msc M}$ in $\Sigma$, which clearly exists, and call it $(\sigma_0, \sigma_1)$. We claim that this is a minimal element for $\preceq$, in fact, assume that $(\sigma^\prime_0, \sigma_1^\prime) \preceq (\sigma_0, \sigma_1)$. Then in particular $(\sigma^\prime_0, \sigma_1^\prime)  \preceq (\bar \mu_0, \bar \mu_1) $ so that $(\sigma^\prime_0, \sigma_1^\prime)   \in \Sigma$; moreover by the first condition of \eqref{eq:relation} there  holds $\sigma_1^{\prime,\pm} \leq \sigma_1^{\pm}$; by minimality for $\mc F$ we deduce $\sigma_1^\prime = \sigma_1$, and therefore $\sigma_0^\prime = \sigma_0$. Therefore $(\sigma_0, \sigma_1)$ is a minimal element for $\preceq$.

Now apply twice Theorem \ref{thm:lagraex} first to $u$ with the pair $(\widehat \mu_0, \widehat \mu_1)$ and then to $\wt u$ (with the reversed flux $\bs g(v) = -f(1-v)$) with the pair $(-\widehat \mu_0, -\widehat \mu_1)$. The new couple of pairs induced by the Lagrangian representations $\bs \omega_h$ and $\bs \omega_e$ (which are \emph{good}, according to Theorem \ref{thm:lagraex}) must coincide with the starting one, by minimality for the relation $\preceq$.
\end{proof}

\begin{remark}
We proved that any minimal pair $(\mu_0, \mu_1)$ for the relation \say {$\preceq$} is simultaneously induced. This proves that if a pair is efficient enough, then it is simultaneously induced. A partial converse holds: if a pair $(\mu_0, \mu_1)$ is simultaneously induced by $\bs \omega_h$, $\bs \omega_e$, then it is \say{efficient} in the sense that 
$$
\mr{supp}\, \mu_0\;  \cup \; \mr{supp}\, \mu_1  \subset \; \partial\, (\mr{hyp}\; u).
$$
\end{remark}

\subsubsection{Good Selection of Curves}
Given a Lagrangian representation $\bs \omega \in \msc M^+(\Gamma)$, we can select a good set of curves on which it is concentrated.  In this direction, with the same proof contained in \cite{marconi_burgers}, we have the following Lemma.
   \begin{lemma}\label{lemma:goodcurves}
       For $\bs \omega_h$-a.e. $\gamma \in \Gamma$ it holds that for $\msc L^1$-a.e. $t \in [0,T]$
       \begin{enumerate}
           \item $(t, \gamma^x(t))$ is a Lebesgue point of $u$;
           \item $\gamma^v(t) < u(t, \gamma^x(t))$
       \end{enumerate}
       We denote by $\Gamma_h$ the set of curves $\gamma \in \Gamma$ such that the two properties above hold. Similarly, for $\bs \omega_e$-a.e. $\gamma \in \Gamma$ it holds that for $\msc L^1$-a.e. $t \in [0,T]$
       \begin{enumerate}
           \item $(t, \gamma^x(t))$ is a Lebesgue point of $u$;
           \item $\gamma^v(t) > u(t, \gamma^x(t))$
       \end{enumerate}
       and we denote by $\Gamma_e$ the set of curves $\gamma \in \Gamma$ such that the two properties above hold. 
\end{lemma}

\section{Structure of the Kinetic Measures}\label{sec:structurekm}
\subsection{Rectifiability of J}
We assume from now on that $u$ is a bounded quasi-entropy solution taking values in $[0,1]$, $u \in \mathbf L^\infty(\Omega, [0,1])$ with flux $f$. This is not restrictive since we consider in any case bounded quasi-entropy solutions.

Throughout this section, we will use the following structural assumption about the nonlinearity of the flux $f$.
\begin{defi}\label{def:fluxgenn}
    We say that a flux $f: [0, 1] \to \mathbb  R^d$ is \textit{weakly genuinely nonlinear} if 
    \begin{equation}\label{eq:genuinon}
        \msc L^1 \Big( \big\{ v \in (0, 1) \; \big| \; \tau + \xi \cdot f^{\prime}(v) = 0 \big\}\Big) = 0, \qquad \forall \; 
        (\xi, \tau)  \in \mathbb S ^{d+1}
    \end{equation}
\end{defi}
In \cite{DOW03}, the structure of the kinetic measures $\mu_1$ in \eqref{eq:kinu}  has been studied in general dimension, when $\mu_0 = 0$. These results directly apply also to the case when a source term $\mu_0$ is present, although some extra care must be used when choosing the representative $(\mu_0, \mu_1) \in \mathbf P_\chi$. We summarize below the results that are obtained directly from \cite{DOW03}, in our setting. We first recall some definitions.
\begin{defi}\label{defi:projset}
In connection to a pair $(\mu_0, \mu_1) \in \mathbf P_{\chi}$,
\begin{enumerate}
    \item we denote by $\nu_0, \nu_1$ the $(t, x)$ marginals of the total variation of $\mu_1, \mu_0$:
\begin{equation}\label{eq:projmeas}
    \nu_0 \doteq [p_{t, x}]_{\sharp} | \mu_0| , \qquad \nu_1 \doteq  [p_{t, x}]_{\sharp} |\mu_1|
\end{equation}
where $p_{t, x} : \Omega\times (0,1)\to \Omega$ is the natural projection on the $(t, x)$ variables. We also let $$\nu \doteq \nu_0 + \nu_1.$$

\item we denote by  $\mathbf J \subset \Omega$ the set of points $(t, x)$ of positive $\msc H^{d}$ density of $\nu$:
\begin{equation}\label{eq:setJ}
    \mathbf  J \doteq \Bigg\{(t, x) \in (0, T) \times \mathbb R^d \; \Big| \; \limsup_{r \downarrow 0} \frac{\nu\left(\mr B_r(t,x)\right)}{r^{d}} > 0\Bigg\}
\end{equation}

\item we say that $u: \Omega \to (0, 1)$ has vanishing mean oscillation at a point $(t, x)$ if 
\begin{equation}
    \lim_{r \downarrow 0} r^{-d-1}\int_{\mr B_r\left(t, x\right)} |u(s, y) - \ol u_r(t, x)| \dif s \dif y = 0
\end{equation}
where $\ol u_r(t, x)$ is the mean of $u$ in the ball $\mr B_r(t,x)$.
\item  Let $J \subset \mathbb{R}^{d+1}$ be a $d$-rectifiable set with unit normal $\vec{\bs{  n}}$. We call two Borel functions $u^-, u^+ : J \to \mathbb{R}$ left and right traces of $u$ on $J$ with respect to $\vec{\bs{  n}}$ if, for $\msc{H}^{d}$-a.e. $(s,y) \in J$,
$$
\lim_{r \downarrow 0} \frac{1}{r^n} \left( \int_{B^-_r(s,y)} |\bs u(t,x) - u^-(s,y)| \, \dif  t\dif x + \int_{B^+_r(s,y)} |\bs u(t,x) - u^+(s,y)| \,  \dif  t\dif x \right) = 0,
$$
where $B^\pm_r(s,y) := \{ (t,x) \in B_r(s,y) \mid \pm \big((t,x)-(s,y) \big) \cdot\vec{\bs{  n}}(s,y) > 0 \}$.
\end{enumerate}
\end{defi}
In our context, the result of \cite{DOW03}, yields
\begin{thm}[\cite{DOW03}]\label{thm:dimdthm}
   Let $d \in \mathbb N$ and $f$ be weakly genuinely nonlinear. Let $u$ be a quasi-entropy solution of \eqref{eq:conslaw}, and let $(\mu_0, \mu_1)$ be a minimal pair in $\mathbf P_{\chi}$. Then the set $\mathbf J$ in \eqref{eq:setJ}  is d-rectifiable and 
    \begin{enumerate}
        \item  $u$ has vanishing mean oscillation at every point $(t,x) \in \mathbf J^c$, the complement of $\mathbf J$,
        \item $u$ has left and right traces on $\mathbf J$.
   \item $\mu_\eta\llcorner \mathbf J= ((\eta(u^+),q(u^+)) - (\eta(u^-),q(u^-)))\cdot \vec{\bs n} \mathscr H^d\llcorner \mathbf J$, where $u^\pm$ denotes the traces on $J$ and $\vec{\bs n}$ denotes the normal to $\mathbf J$.
    \end{enumerate}
\end{thm}

For $BV$ solutions (1) and (3) can be improved to
\begin{enumerate}
\item[(1')] every $(t,x)\notin \mathbf J$ is a Lebesgue point;
\item[(3')] $\mu_\eta= ((\eta(u^+),q(u^+)) - (\eta(u^-),q(u^-)))\cdot \vec{\bs n}\mathscr H^d \llcorner \mathbf J$.
\end{enumerate}

\begin{remark}
    One may wonder where the minimality assumption of $(\mu_0, \mu_1)$ enters in the proof of Theorem \ref{thm:dimdthm}: the authors of \cite{DOW03} consider an equivalent kinetic formulation, for quasi-entropy solutions, which reads:
    \begin{equation}
        \chi_t + f^\prime(v)\cdot \nabla_x \chi = \partial_v \wt \mu
    \end{equation}
    where $\wt \mu$ is possibly non-compactly supported in $v$ to account for a source term. The link with our $\mu_0, \mu_1$ (which are compactly supported in $v$) is 
    $$
    \wt \mu(t,x,v): = \mu_1(t,x,v) + \int_{-\infty}^v \dif \mu_0(t,x, v)
    $$
    where the second measure in the right hand side  denotes the primitive in the $v$ variable of the measure $\mu_0$.
 In \cite{DOW03}, the rectifiable jump set, let us call it $\wt {\mathbf J}$, is defined as the set of positive $\mathscr H^d$-density of the measure $$\wt \nu := (p_{t,x})_\sharp |\wt \mu| \llcorner (\mathbb R^{d+1} \times(-\infty, |u|_\infty)).$$ Our $\mathbf J$ is instead the set of positive $\mathscr H^d$-density of the measure $\nu:=(p_{t,x})_\sharp(|\mu_1| +|\mu_0|)$. Therefore to obtain Theorem \ref{thm:dimdthm} from the results of \cite{DOW03} we only need to show that 
 $$
 \mathbf J =  \wt{\mathbf J}  \qquad \text{up to an $\mathscr H^d$-negligible set}.
 $$
To prove this, it is enough to prove that $\nu \ll \wt \nu$ (notice that the other $\wt \nu \ll \nu$ is trivial). By contradiction, assume that there is a measurable set $A \subset \mathbb R^{d+1}$ with positive measure such that 
$$
\nu(A) > 0, \qquad \wt \nu(A) = 0.
$$
But this means that $$\Big( \mu_1 + \int_{-\infty}^v \mu_0(t,x,v) \dif v\Big)\llcorner A = 0$$ that implies $$ \partial_v \mu_1\llcorner A + \mu_0\llcorner A = 0$$ and therefore we can replace the pair $(\mu_0, \mu_1)$ with another pair $(\mu_0^\prime, \mu_1^\prime) \in \mathbf P_\chi$ defined by
$$
\mu_1^\prime = \mu_1 \llcorner A^c, \qquad \mu_0^\prime = \mu_0 \llcorner A^c.
$$
Clearly $(\mu_0^\prime, \mu_1^\prime) \preceq (\mu_0, \mu_1)$, and this contradicts the minimality of $\mu_0, \mu_1$.
\end{remark}

Establishing (1') for general weakly genuinely nonlinear fluxes is an open problem at the time of writing, but in  \cite{Si18} the author considered the case of entropy solutions  with a power-type nonlinearity assumption on $f$, and in this setting he proved that every point $(t,x)\notin J$ is a continuity point, providing a 
positive answer about (1') in this particular case. Moreover, in \cite{marconi_structure} the author showed that, for general finite entropy solutions, the set of non-Lebesgue points has Hausdorff dimension at most $d$, also providing a partial answer to (1). Also, for entropy solutions in one space dimension both (1') and (3') have affirmative answers: see \cite{BM_structure}.

In the following section, we prove that property (3') holds for quasi-entropy solutions in one space dimension.

\subsection{The One-Dimensional Case}
The aim of this section is to extend the rectifiability result \cite{marconi_burgers} to general genuine nonlinear fluxes in one space dimension. In particular, we will prove the following Theorem.

\begin{thm}\label{thm:1drect}
  Let $d = 1$ and $f$ satisfy \eqref{eq:genuinon}. Let  $u \in \mathbf L^\infty(\Omega; [0,1])$ be a quasi-entropy solution of \eqref{eq:conslaw}, and let $(\mu_0, \mu_1)$ be a minimal pair in $\mathbf P_{\chi}$. Then, in addition to Theorem \ref{thm:dimdthm}, $\nu_1$ in \eqref{eq:projmeas} is concentrated on $\mathbf J$, i.e.
  \begin{equation}\label{eq:concentration}
      \nu_1 = \nu_1 \llcorner \mathbf J
  \end{equation}
\end{thm}

We will actually prove that $\nu_1$ is concentrated on a $1$-rectifiable set, it is then immediate to see that the set coincides with the set $\mathbf J$ thanks to formula (3) of Theorem \ref{thm:dimdthm}.

For convenience, recall that $\Omega = (0,T) \times \mathbb R$, and in the following we denote
$
 X = \Omega \times (0, 1).
$
The strategy of the proof is as follows. From Proposition \ref{prop:goodpairs} we know that if $(\mu_0, \mu_1)$ is a minimal pair, then it is simultaneously induced by some $\bs \omega_h, \bs \omega_e$ \emph{good} Lagrangian representations of  the hypograph and the epigraph of $u$, respectively.
The concentration on $\mathbf J$ of the full measure $\nu_1$ will follow by the concentration on $\mathbf J$ of both the measures $[p_{t, x}]_{\sharp} \mu_1^+$ and $[p_{t, x}]_{\sharp} \mu_1^-$. Therefore first we prove the result for $[p_{t, x}]_{\sharp} \mu_1^+$, and the other case will follow by symmetry.  

The proof is established in five steps.

\vspace{0.5cm}

\textbf{1.}  We introduce the measures  $\bs \omega_h \otimes \mu_1^{\gamma, +}$,  $\bs \omega_e \otimes \mu_1^{\gamma, -} \in \msc M(\Gamma \times X)$, i.e. the measures defined as, for Borel sets $G \in \mc B(\Gamma)$, $A \in \mc B(X)$, 
    \begin{equation}
    \begin{aligned}
        \bs \omega_h \otimes {\mu_1^{\gamma, +}}(G \times A) : = \int_{G}\mu_1^{\gamma, +}(A) \dif \bs \omega_h(\gamma), \\
         \bs \omega_e \otimes \mu_1^{\gamma, -} (G \times A) : = \int_{G}\mu_1^{\gamma, -}(A) \dif \bs \omega_e(\gamma)
         \end{aligned}
    \end{equation}
    where $\mu_1^{\gamma \pm}$ are defined in \eqref{eq:mu1def}.
    Since $(\mu_0, \mu_1)$ is simultaneously induced by $\bs \omega_h$ and $\bs \omega_e$, it is possible to choose a \textit{transport plan} $\pi^+ \in \msc M((\Gamma \times X )^2)$ which transports $\bs \omega_h \otimes \mu_1^{\gamma, +}$ to $\bs \omega_e \otimes \mu_1^{\gamma, -}$ and which is concentrated on the set
    \begin{equation}\label{eq:setmcG}
        \begin{aligned}
    \mc G : =  \Big\{&  (\gamma, t, x, v, \gamma^{\prime}, t, x, v) \in (\Gamma \times X )^2 \; \big|  \\
    & \gamma^x(t) = \gamma^{\prime, x}(t), \quad v \in [\gamma^v(t+), \gamma^v(t-)] \cap   [\gamma^{\prime,v}(t-),\gamma^{\prime,v}(t+)] \Big\}
            \end{aligned}
    \end{equation}
This follows observing that, because $(\mu_0, \mu_1)$ is simultaneously induced by $\bs \omega_h$ and $\bs \omega_e$,
we have
\begin{equation*}
    \int_{\Gamma}\mu_1^{\gamma, -}\dif \bs \omega_h(\gamma)=
    \mu_1^-=
     \int_{\Gamma}\mu_1^{\gamma, +}\dif \bs \omega_e(\gamma),
     \qquad
     \int_{\Gamma}\mu_1^{\gamma, +}\dif \bs \omega_h(\gamma)=
    \mu_1^+=
     \int_{\Gamma}\mu_1^{\gamma, -}\dif \bs \omega_e(\gamma),
\end{equation*}
and applying
the following Lemma. 
\begin{lemma}[\cite{marconi_burgers}, Lemma 8]
    Denote by $P_1, P_2: (\Gamma \times X)^2 \to (\Gamma \times X)$ the standard projections. Then there exists a plan $\pi^+ \in \msc M((\Gamma \times X)^2)$ with marginals 
    \begin{equation}
    \begin{aligned}
        & [P_1]_{\sharp} \pi^+ = \bs \omega_h \otimes \mu_1^{\gamma, +},\\
        & [P_2]_{\sharp} \pi^+ = \bs \omega_e \otimes \mu_1^{\gamma, -}
        \end{aligned}
    \end{equation}
    concentrated on the set $\mc G$ defined in  \eqref{eq:setmcG}.
\end{lemma}

\vspace{0.5cm}
\textbf {2.}
In this step, we decompose the plan $\pi^+$ constructed in Step 1 into  a countable sum of components.  First, consider the set $I =\{v \; : \; f^{\second}(v) =0\}$, and its complement $(0, 1) \setminus I$. Since  $(0, 1) \setminus I$ is an open set, we can write it as the union of countably many disjoint intervals $A_l$:
$$
(0, 1) \setminus I = \bigcup_{l=1}^{\infty} A_l, \qquad A_l \subset (0, 1) \quad \text{disjoint intervals}
$$
Then we will decompose $\pi^+$ as
\begin{equation}
    \pi^+ = \sum_{l \in\mathbb N} \pi^+_l + \pi^+_{J}
\end{equation}
where, loosely speaking, $\pi^+_l$ is supported in the set  $(\gamma,t,x, v,  \gamma^{\prime}, t, x, v^{\prime})$  where at the point $(t, x)$ the curves $\gamma, \gamma^{\prime}$ satisfy 
$$
\{ \gamma^v(t+), \gamma^v(t-), \gamma^{\prime,v}(t+), \gamma^{\prime,v}(t-)\} \subset A_l;
$$
 while $\pi^+_J$ represents the remaining curves. In the following of this step we formalize this discussion.

By definition, the measure $\bs \omega_h \otimes \mu_1^{+, \gamma}$ (recall also \eqref{eq:mu1def}) is concentrated  on the set 
\begin{equation}
\begin{aligned}
    \mc G_h^+ : = \Big\{(\gamma, t, x, v) \in \Gamma \times X \; \Big| \; \gamma^x(t) = x, \; v \in [\gamma^v(t+), \gamma^v(t-)]\Big\}.
    \end{aligned}
\end{equation}
Analogously, the measure $\bs \omega_e \otimes \mu_1^{-, \gamma}$  is concentrated  on the set 
\begin{equation}
\begin{aligned}
    & \mc G_e^- : = \Big\{(\gamma^{\prime}, t^{\prime}, x^{\prime}, v^{\prime}) \in \Gamma \times X \;\Big| \; \gamma^x(t^\prime) = x^\prime, \; v^\prime \in [\gamma^v(t^\prime-), \gamma^v(t^\prime+)]\Big\}.
    \end{aligned}
\end{equation}
We define the sets, for each $l \in \mathbb N$,
\begin{equation}
\begin{aligned}
    & \mc G^+_{h, l} : = \Big\{(\gamma, t, x, v) \in \mc G^+_h \; \Big| \; \gamma^v(t-), \gamma^v(t+) \in A_l\Big\},\\
& \mc G^-_{e, l} : = \Big\{(\gamma^{\prime}, t^{\prime}, x^{\prime}, v^{\prime}) \in \mc G^-_e \; \Big| \; \gamma^v(t-), \gamma^v(t+) \in A_l\Big\}.
    \end{aligned}
\end{equation}
and the sets
\begin{equation}
    \begin{aligned}
  & \mc G^+_{h, J} : = \Big\{(\gamma, t, x, v) \in \mc G^+_h \; \Big| \; \gamma^v(t-)> \gamma^v(t+) \Big\},\\
& \mc G^-_{e, J} : = \Big\{((\gamma^{\prime}, t^{\prime}, x^{\prime}, v^{\prime})\in \mc G^-_e \; \Big| \; \gamma^v(t^\prime-)< \gamma^v(t^\prime+) \Big\}.
    \end{aligned}
\end{equation}
Finally we define
\begin{equation}
\begin{aligned}
    & \pi^+_l := \pi^+ \llcorner (  \mc G^+_{h, l} \times \mc G^-_{e, l}), \\
    & \pi^+_J : = \pi^+ \llcorner \Bigg( \Big((\mc G_h^+\times \mc G^-_{e, J} ) \cup ( \mc G^+_{h, J} \times \mc G_e^-)\Big) \setminus \bigcup_{\ell}  ( \mc G^+_{h, l} \times \mc G^-_{e, l}) \Bigg).
    \end{aligned}
\end{equation}
By Step 1 it holds
$$
\pi^+ = \sum_{l \in\mathbb N} \pi^+_l + \pi^+_{J} + \pi^+_I
$$
where 

$$
\pi^+_I : = \pi^+\llcorner (\mc G^+_{h,I} \times \mc G^-_{e,I}),
$$
$$
\mc G^+_{h,I} := \Big\{ (\gamma, t, x, v) \in \mc G^+_h \Big| \gamma(t-) = \gamma(t+) \in I\Big\},
$$
$$
\mc G^-_{e,I} := \Big\{ (\gamma^{\prime}, t^{\prime}, x^{\prime}, v^{\prime}) \in \mc G^+_h \Big| \gamma(t^{\prime}-) = \gamma(t^{\prime}+) \in I\Big\}.
$$
We claim that $\pi^+_I = 0$ so that actually 
$$
\pi^+ = \sum_{l \in\mathbb N} \pi^+_l + \pi^+_{J}.
$$
This follows by the following Lemma about functions of bounded variation (see e.g. \cite{AFP_book}).
\begin{lemma}
    Let $v:(a, b) \to (0, 1)$ be a $BV$ function. Then, if $\wt D v$ is the diffuse part of the measure $Dv$, for any set $I \subset (0, 1)$ of zero $\msc L^1$ measure, it holds
    $$
    \wt D v(v^{-1}(I)) = 0
    $$
\end{lemma}

\vspace{0.5cm}
\textbf{3.}
Using the construction of Step 2, we now prove that for every $l \in \mathbb N$, the measure
\begin{equation}
    \nu^+_{1,l}: = [P_{t,x}]_{\sharp} \pi^+_l
\end{equation}
is concentrated on a 1-rectifiable set, where $P_{t,x} : (\Gamma \times X) ^2 \to (0, T) \times \mathbb R$ is the projection on the first two variables $(\gamma, t, x, v, \gamma^{\prime}, t^{\prime}, x^{\prime}, v^{\prime} ) \mapsto (t, x)$.

We fix $l \in \mathbb N$, and prove that $\nu^+_{1,l}$ is concentrated on a 1-rectifiable set. Without loss of generality we can assume that $A_l \subset (0, 1)$ is such that $f^{\second}(v) > 0$ for every $v \in A_l$. The other case is completely symmetric.
We start with two preliminary results, that are proved in \cite{marconi_burgers}. 

The first lemma formalizes the intuitive fact that, in an interval $A_l$ where the flux is strictly convex, generically, Lagrangian curves representing the hypograph cannot cross from the left curves representing the epigraph.
\begin{lemma}\label{lemma:nocrossing}
    Let $\Gamma_h, \Gamma_e$ as in Lemma \ref{lemma:goodcurves}. Let $(\bar \gamma, I_{\bar \gamma}) \in \Gamma_h$ and let $(a, b) \subset I_{\bar \gamma}$ 
 be such that $\bar \gamma^v((a, b) )\subset A_l$. Let moreover $G \subset \Gamma_e$ be a set of curves $(\gamma, I_{\gamma})$ such that there exist $a < s^1_{\gamma} < s^2_{\gamma}< b$ with
\begin{equation}\label{eq:crossingcond}
    \gamma^x(s^1_{\gamma}) > \bar \gamma^x (s^1_{\gamma}), \qquad \gamma^x(s^2_{\gamma}) <  \bar \gamma^x (s^2_{\gamma}), \qquad \gamma^v(s^1_{\gamma}, s^2_{\gamma}) \subset A_l.
\end{equation}
Then 
\begin{equation}
    \bs \omega_e(G) = 0.
\end{equation}
\end{lemma}
The proof of the Lemma can be found in \cite[Proposition 6]{marconi_burgers} in the case of Burgers equation, but it  is the same for strictly convex fluxes, therefore we omit the proof. 

Given $\bar t, \bar x$, let $G^l_{\bar t,\bar x}\subset \Gamma_h$ be the set of curves $(\gamma, I_{\gamma})$ such that $\bar t \in I_{\gamma}$, $\gamma^v(\bar t) \in A_l$, and $\gamma^x(\bar t) < \bar x$. Analogously, let $G^r_{\bar t,\bar x}\subset \Gamma_e$ be the set of curves $(\gamma, I_{\gamma})$ such that $\bar t \in I_{\gamma}$, $\gamma^v(\bar t) \in A_l$, and $\gamma^x(\bar t) > \bar x$.

Now we construct the candidate Lipschitz curves on which $\nu^+_{1,l}$ is concentrated, \cite[Corollary 7]{marconi_burgers}.
\begin{lemma}
 To each $(\gamma, I_{\gamma}) \in \Gamma_h$ associate the time $t_{\gamma}^l \in (\bar t, t^2_{\gamma}]$ defined as 
$$
t^l_{\gamma}: = \sup \big\{t \in [\bar t, t^2_{\gamma}) \; \big| \;\gamma^v(s) \in A_l \; \forall \; s \in (\bar t, t) \big\}
$$
where we understand that the sup of the empty set is $\bar t$.
Define the  curve
$$
    \tilde f_{\bar t, \bar x}^l(t) : = \sup_{\{\gamma \in G^l_{\bar t, \bar x} \; : \;  t \in (\bar t, \;  t^l_{\gamma})\}} \gamma^x(t), \qquad t \in [\bar t, T)
$$
and its upper Lipschitz envelope $f^l_{\bar t, \bar x} : [\bar t, T) \to \mathbb R$
\begin{equation}
    f^l_{\bar t, \bar x}(t) : = \inf \big\{g(t) \; \big| \;  g:[\bar t, T] \to \mathbb R \; \text{is $|f^{\prime}|_{\infty}$-Lipschitz and $g \geq \tilde f_{\bar t, \bar x}^l$ in $ [\bar t, T)$}\big\}. 
\end{equation}
Then, for every $t \in [\bar t, T)$, it holds
\begin{equation}
    \begin{aligned}
        &  \bs \omega_h\Big( \big\{ \gamma \in G^l_{\bar t, \bar x} \; \big| \; \gamma^x(t) > f^l_{\bar t, \bar x}(t), \quad t \in [\bar t, t^l_{\gamma})] \big\}\Big) = 0\\
        & \bs \omega_e\Big( \big\{ \gamma \in G^r_{\bar t, \bar x} \; \big| \; \gamma^x(t) < f^l_{\bar t, \bar x}(t), \quad t \in [\bar t, t^l_{\gamma}) \big\}\Big) = 0.
    \end{aligned}
\end{equation}
\end{lemma}
The following Lemma is a general result about functions of bounded variation (see \cite{AFP_book}).
\begin{lemma}\label{lemma:bvlemma}
Let $v : (a, b) \to \mathbb R$ be a $BV$ function and denote by $D^- v$ the negative part of the measure $D v$. Then for $\tilde D^- v$-a.e. $\bar x \in (a, b)$ there exists a $\delta > 0$ such that
    $$
    v(x) > v(\bar x) \quad \forall \; x \in (\bar x-\delta, \bar x), \qquad v(x) < v(\bar x) \quad \forall \; x \in (\bar x, \bar x +\delta)
    $$
\end{lemma}

Now that we have all the elements, we divide the proof of Step 3 into further substeps.

\vspace{0.5cm}
\textbf{3.1.} Fix $(\bar t, \bar x) \in (0, T) \times \mathbb R$. Define the sets $\mc G_{\bar t, \bar x}^{h,l}$, $ \mc G^{e,l}_{\bar t, \bar x} \subset \Gamma \times X$ as 
\begin{equation}
\begin{aligned}
    & \mc G^{h,l}_{\bar t, \bar x} : = \big\{(\gamma, t, x, v) \in \Gamma_h \times X  \; : \; t \in (\bar t, t^{l, \bar t}_{\gamma}), \quad \gamma^x(\bar t) < \bar x\big\} \\
& \mc G^{e,l}_{\bar t, \bar x} : = \big\{(\gamma, t, x, v) \in \Gamma_e \times X  \; : \; t \in (\bar t, t^{l, \bar t}_{\gamma}),  \quad \gamma^x(\bar t) > \bar x\big\} .
    \end{aligned}
\end{equation}
Define the measure 
$$
\pi^+_{l,\bar t, \bar x} := \pi^+_l \llcorner \big( \mc G^{h,l}_{\bar t, \bar x} \times \mc G^{e, l}_{\bar t, \bar x}\big)
$$
The main contribution of this step is to show that, letting 
$$\msc F_{\bar t, \bar x}^l : = \{(t,x) \; | \; x = f^l_{\bar t, \bar x}(t), \quad t \in (\bar t, T)\}$$
then
\begin{equation}\label{eq:step31}
    \text{the measure $(P_{t,x})_{\sharp}\pi^+_{l,\bar t, \bar x}$ is concentrated on $\msc F_{\bar t, \bar x}^l $}.
\end{equation}
Let $\Omega^{l,\pm}_{\bar t, \bar x}$ be the two connected components  of $(\bar t, T) \times \mathbb R \setminus \msc F^l_{\bar t, \bar x}$, the left ($-$) and the right ($+$) one, respectively. By definition of $\mc G^{h,l}_{\bar t, \bar x}$ and of $\msc F^l_{\bar t, \bar x}$, and by (4.20), there holds
\begin{equation}\label{eq:step311}
\bs \omega_h \otimes \mu_1^{\gamma,+}\llcorner \mc G^{h,l}_{\bar t, \bar x} \big( \Gamma_h \times \Omega^{l,+}_{\bar t, \bar x}\times (0, 1) \big) = 0.
\end{equation}
Moreover, one has 
\begin{equation}\label{eq:step312}
\begin{aligned}
    [P_{t,x}]_{\sharp}\pi^+_{l,\bar t, \bar x} &  \leq [P_{t,x}]_{\sharp}\pi^+_{l} \llcorner \big( \mc G^{h,l}_{\bar t, \bar x}  \times (\Gamma \times X)\big)\\
    & \\
    & = [p_{t, x}]_{\sharp} \left(  \int_{\Gamma_h} \mu_1^{\gamma,+}\llcorner \mc G^{h,l}_{\bar t, \bar x}   d \bs\omega_h  \right). \\
    \end{aligned}
\end{equation}
It follows from \eqref{eq:step311}, \eqref{eq:step312} that 
\begin{equation}
     [P_{t,x}]_{\sharp}\pi^+_{l,\bar t, \bar x}(\Omega^{l,+}_{\bar t, \bar x}) = 0.
\end{equation}
In an entirely similar way, we prove
\begin{equation}
[P_{t^{\prime},x^{\prime}}]_{\sharp}\pi^+_{l,\bar t, \bar x}(\Omega^{l,-}_{\bar t, \bar x}) = 0.
\end{equation}
Finally, we have that $[P_{t^{\prime},x^{\prime}}]_{\sharp}\pi^+_{l,\bar t, \bar x} =  [P_{t,x}]_{\sharp}\pi^+_{l,\bar t, \bar x}$, because $\pi^+$ is supported in $\mathcal G$ and therefore also $\pi_l^+$ is supported in $\mathcal G$.
Therefore we conclude that 
\begin{equation}
    [P_{t,x}]_{\sharp} \pi^+_{l,\bar t, \bar x}( (\Omega^{l,+}_{\bar t, \bar x} \cup \Omega^{l,-}_{\bar t, \bar x})) = 0
\end{equation}
which means that $[P_{t,x}]_{\sharp} \pi^+_{l,\bar t, \bar x}$ is concentrated on $\msc F^l_{\bar t, \bar x}$.

\vspace{0.5cm}
\textbf{3.2.} We prove that for $\pi^+_l$-a.e. pair $(\gamma, t, x, v, \gamma^{\prime}, t^{\prime}, x^{\prime}, v^{\prime})$ there exists a $\delta > 0$ such that 
\begin{enumerate}
    \item for every $s \in [t-\delta, t)$ it holds $\gamma^x(s)< \gamma^{\prime x}(s)$, 
    \item $\gamma^v(s) \in A_l$ for all $s \in (t-\delta, t]$,
    \item$\gamma^{\prime, v}(s) \in A_l$ for all $s \in (t^{\prime}-\delta, t^{\prime}]$.
\end{enumerate}
In order to prove it, we proceed as follows. By definition of $\pi_l^+$, it holds
\begin{equation}
\begin{aligned}
    \text{for $\pi^+_l$-a.e. pair}&  \text{ $(\gamma, t, x, v, \gamma^{\prime}, t^{\prime}, x^{\prime}, v^{\prime})$  it holds:} \\
    & \begin{cases}
          \gamma^v(t-)\geq\gamma^v(t+) \; \text{and} \; \gamma^v(t-), \gamma^v(t+) \in A_l  \\
         \gamma^{\prime,v}(t-)\leq\gamma^{\prime,v}(t+)\;  \text{and} \; \gamma^{\prime,v}(t-), \gamma^{\prime,v}(t+) \in A_l
    \end{cases}
    \end{aligned}
\end{equation}
Therefore there exists a $\delta_{2, 3} > 0$ such that (2), (3), are satisfied.  To prove (1), we first prove the following claim:
\begin{equation} \label{eq:muminuscla}
    \text{for $\mu^-_{\gamma}$-a.e. $(t, x, v)$, there exists a $\delta > 0$  such that $\gamma^v(s) > v$ for every $s \in (t-\delta, t)$}
\end{equation}
An application of Lemma \ref{lemma:bvlemma} provides a $\tilde D^-\gamma$-negligible subset $N_{\gamma} \subset I_{\gamma}$ such that 
 for every $t \in I_{\gamma} \setminus  N_{\gamma}$ there exists a $\delta$ such that $\gamma^v(s) > \gamma^v(t+)$ for every $s \in (t-\delta, t)$.
Moreover, for every $t \in I_{\gamma}$ in which $\gamma^v$ has a negative jump, for each $v \in [\gamma^v(t+), \gamma^v(t-))$ there exists a $\delta > 0$ such that for every $s \in (t-\delta, t)$ it holds $\gamma^v(s) > v$. Let 
$$
E_{\gamma} = \{(t, x, v) \; : \; \gamma^v(t-) = v> \gamma^v(t+) \}
$$
Since $E_{\gamma}$ is at most countable and $\mu_1^{\gamma,+}$ has no atoms, it follows that $\mu_1^{\gamma,+} \big( E_{\gamma}\big) = 0$. Therefore
$$
\mu_1^{\gamma,+} \big( E_{\gamma} \cup (\mathbb I, \gamma)(N_{\gamma})\big) = 0
$$
and \eqref{eq:muminuscla} is proved. Therefore, using also (2), we obtain the following statement: for $\bs \omega_h \otimes \mu_1^{\gamma,+}$-a.e. $(\gamma, t, x, v)$, there exists a $\delta_h> 0$ ($\delta_h < \delta_2$)  such that 
\begin{equation}
    \text{ for every $s \in (t-\delta_h, t)$ it holds $\gamma^v(s) > v$, and hence $\gamma^x(s) < x-(t-s) f^{\prime}(v)$ } 
\end{equation}
where we used the strict convexity of $f$ in $A_l$ and the characteristic equation for $\gamma$ \eqref{eq:charaeqla}. In an entirely analogous way we prove the symmetric statement: for $\bs \omega_e \otimes \mu_{\gamma}^-$-a.e. $(\gamma, t, x, v)$, there exists a $\delta_e> 0$ ($\delta_e < \delta_3$) such that 
\begin{equation}
    \text{ for every $s \in (t-\delta_h, t)$ it holds $\gamma^v(s) < v$, and hence $\gamma^x(s) > x-(t-s) f^{\prime}(v)$ } 
\end{equation} 
To conclude is sufficient to notice that since $\pi^+$ is concentrated on $\mc G$, then (1) holds with $\delta = \min (\delta_h, \delta_e)$ for $\pi^+_l$-a.e. $(\gamma, t, x, v, \gamma^{\prime}, t^{\prime}, x^{\prime}, v^{\prime})$.

\vspace{0.5cm}
\textbf{3.3.} From Step 3.2, we deduce that for $\pi^+_l$-a.e. pair $(\gamma, t, x, v, \gamma^{\prime}, t^{\prime}, x^{\prime}, v^{\prime})$, there exists a rational pair $(\bar t, \bar x) \in ((0, T) \cap \mathbb Q) \times \mathbb Q$ such that 
$$
(\gamma, t, x, v, \gamma^{\prime}, t^{\prime}, x^{\prime}, v^{\prime}) \in \mc G^{h,l}_{\bar t, \bar x} \times \mc G^{e,l}_{\bar t, \bar x} 
$$
This shows that $\pi^+_l$ is concentrated on 
$$
\bigcup_{\substack{\bar t \in (0, T) \cap \mathbb Q \\ \bar x \in \mathbb Q}} \mc G^{h,l}_{\bar t, \bar x} \times \mc G^{e,l}_{\bar t, \bar x} 
$$
Since holds
\begin{equation}
    \begin{aligned}
          [P_{t,x}]_{\sharp} \pi^+_l&  = \sum_{\substack{\bar t \in (0, T) \cap \mathbb Q \\ \bar x \in \mathbb Q}} [P_{t,x}]_{\sharp} \Big( \pi^+_l \llcorner \mc G^{h,l}_{\bar t, \bar x} \times \mc G^{e,l}_{\bar t, \bar x} \Big)
    \end{aligned}
\end{equation}
by Step 3.1 it follows that $[P_{t,x}]_{\sharp} \pi^+_l = \nu_{1,l}^+$ is concentrated on the 1-rectifiable set 
$$
\bigcup_{\substack{\bar t \in (0, T) \cap \mathbb Q \\ \bar x \in \mathbb Q}} \msc F^l_{\bar t, \bar x} \subset (0, T) \times \mathbb R.
$$

\vspace{0.5cm}
\textbf{4.}
As a final step, we prove that the remaining part 
\begin{equation}
    \nu^+_{1,J} : =  [P_{t, x}]_{\sharp} \pi_J^+
\end{equation}
is concentrated on the set ${\mathbf J}$ of Theorem \ref{thm:dimdthm}, and therefore is concentrated on a 1-rectifiable set.

 For $\bar v \in (0, 1)$, we introduce the following functions which measure the nonlinearity of $f$ near a point $\bar v$. For $\delta > 0$, define
$$
\mathfrak h^-(\bar v, \delta) : = \max \Big\{ h > 0 \; \Big| \; \msc L^1\big( (\bar v-\delta, \bar v) \cap \{ v \; : \; |f^{\prime}(\bar v) -f^{\prime}(v) | \geq 2 h \}\big) \geq h \Big\} 
$$
$$
\mathfrak h^+(\bar v, \delta) : = \max \Big\{ h > 0 \; \Big| \; \msc L^1\big( (\bar v, \bar v+\delta) \cap \{ v \; : \; |f^{\prime}(\bar v) -f^{\prime}(v) | \geq 2 h \}\big) \geq h \Big\}. 
$$
If the flux is weakly genuinely nonlinear in the sense of Definition \ref{def:fluxgenn}, it holds
\begin{equation}
    0 < \mathfrak h^{\pm}(\bar v, \delta) < \delta, \qquad \delta > 0
\end{equation}
For example, if $f(v) = v^2/2$, then for every $\bar v \in (0, 1)$ one has $\mathfrak h^{\pm}(\bar v, \delta) =  \delta/3$.

We have the following Lemma, proved with the same techniques of \cite[Lemma~4.6]{Mar23}.
\begin{lemma}\label{lemma:hypreg}
    Let $(\gamma, I_{\gamma}) \in \Gamma_h$,  let $\bar t \in \bar I_{\gamma}$ and set $\bar x = \gamma^x(\bar t)$. Let 
    $$\bar v = \max\, [\gamma^v(\bar t\pm)],
$$
    where 
    $$
    [\gamma^v(\bar t\pm)]=
    \begin{cases}
      \{\gamma^v(\bar t+),\, \gamma^v(\bar t-)\}
        \  &\text{if}\quad  \bar t \in I_{\gamma},
        \\
        \{\gamma^v(\bar t+)\}
        \   &\text{if}\quad  \bar t = \inf I_\gamma,
        \\
        \{\gamma^v(\bar t-)\}
        \   &\text{if}\quad  \bar t = \sup I_\gamma.
    \end{cases}
    $$
    
    Then there exists a constant $c$ depending only on $\Vert f^{\second}\Vert_{\infty}$ and $\delta_1$ depending on $\gamma$ such that for all $\delta < \delta_1$ at least one of the following holds true.
    \begin{equation}
    \begin{aligned}
       &  \liminf _{r \downarrow 0} \frac{\msc L^2 \big\{(t, x) \in B_{2r}(\bar t, \bar x) \; \big| \; u(t,x) > \bar v-\delta\big\}}{r^2} > c \cdot \mathfrak h^-(\bar v, \delta)\\
   & \limsup_{r \downarrow 0} \frac{\nu_0(B_{2r}(\bar t, \bar x))}{r} > c \cdot \mathfrak h^-(\bar v, \delta)^2\\
   & \limsup_{r \downarrow 0} \frac{\nu_1(B_{2r}(\bar t, \bar x))}{r} > c \cdot \mathfrak h^-(\bar v, \delta)^3.
        \end{aligned}
\end{equation}
\end{lemma}

\begin{proof}
Without loss of generality, we assume that $\bar v = \bar \gamma^v(\bar t-)$ and that there exists $\delta$ such that $(\bar t-\delta, \bar t) \subset I_\gamma$. We let $\delta > 0$ be such that for every $t \in (\bar t-\delta, \bar t)$ it holds
$$
|f^{\prime}(\bar \gamma^v(t)) -f^{\prime}(\bar v) | < \mathfrak h^-(\bar v, \delta)/2, \qquad \bar \gamma^v(t) > \bar v-\mathfrak h^-(\bar v, \delta)/2
$$
Moreover, since $\exists \lim_{t\to \bar t^-}\bar \gamma_x'(t)=f'(\bar v)$, then for $\bar r$ small, and for every $r < \bar r$, the curve $(t^1_{\bar \gamma}, \bar t)  \ni t \mapsto (t, \bar \gamma^x(t))$ has a unique intersection with $\partial B_r((\bar t, \bar x))$, at a point that we call $t_r$. 

We start by noting that an interval $J \subset (\bar v-\delta, \bar v- \mathfrak h^-(\bar v, \delta)/4)$ of length $\sim  \mathfrak h^-(\bar v, \delta)$ and in which $f^{\prime}$ is distant at least $\mathfrak h^-(\bar v, \delta)$ from $f^{\prime}(\bar v)$:
 \begin{equation}\label{eq:Iilowebound}
 \begin{aligned}
 &|f^{\prime}(\bar v)-f^{\prime}(v)| >  \mathfrak h^-(\bar v, \delta), \qquad \forall v \in J\\
    & |J| \geq \frac{1}{\|f^{\second}\|_{\infty}}  \mathfrak h^-(\bar v, \delta).
     \end{aligned}
 \end{equation}

In fact, pick any point $\tilde v \in (\bar v-\delta, v)$ such that $f^\prime(v) \geq 2h$. Let $J$ be the connected component of the set $\{v \; | \; |f^\prime(v) - f^\prime(\bar v)| > \mathfrak h^-(\bar v, \delta)\}$ to which $\tilde v$ belongs. Then, the length of $J$ must be at least 
$$
|J| \geq \frac{1}{\|f^\second\|_{\infty}} \mathfrak h^-(\bar v, \delta).
$$

Since $\bar \gamma \in \Gamma_h$ and $\bar \gamma^v(t) > \bar v- \mathfrak h^-(\bar v, \delta)/4$ for $t \in (t_1-\bar \delta, t_1)$, it holds, for some $\varepsilon>0$ possibly depending on $r$, that 
\begin{equation}
\begin{aligned}
    & \msc L^2 \Bigg\{ (t, x) \in S^\gamma_{\varepsilon, r} \; \Big| \; u(t, x) > \bar v- \frac{\mathfrak h^-(\bar v, \delta)}{4}\Bigg\} \geq \varepsilon r, \quad \text{where $S^\gamma_{\varepsilon, r}   : = (\mr{id}, \gamma^x)((t_r, \bar t)) + B_{\varepsilon}(0)$}.
    \end{aligned}
\end{equation}
For every $(\gamma, I_{\gamma}) \in \Gamma$ consider the nontrivial interiors $(t_{j}^{\gamma, -},t_{j}^{\gamma, +})_{j=1}^{N_{\gamma}}$ of the connected components of $(\gamma^v)^{-1}(J)$ which intersect
$$
(\mr{id}_t, \gamma)^{-1}\big( S^\gamma_{\varepsilon, r} \times \bar J \big)  \subset I_{\gamma}
$$
where $\bar J$ is the central interval of $J$ of length $|J|/3$. Notice that we have the estimate
\begin{equation}\label{eq:Ngamma}
N_{\gamma} \leq 1+ \frac{3}{|J|}\mr{Tot.Var.}\gamma^v
\end{equation}
For every $j \in \mathbb N$, consider the set 
$$
\Gamma_j : = \{(\gamma, I_{\gamma})\; : \; N_{\gamma} \geq j\}
$$
and consider the measurable restriction  map 
$$
R_j : \Gamma_j \to \Gamma, \quad (\gamma, I_{\gamma} ) \mapsto (\gamma, (t_{j}^{\gamma,-}, t_{j}^{\gamma,+}))
$$
Define the measure 
$$
\tilde \omega_h : = \sum_{j = 1}^{\infty} (R_j)_{\sharp} (\omega_h \llcorner \Gamma_j)
$$
which is finite because of the estimate \eqref{eq:Ngamma} we have for some constant $K> 0$
$$
\|\tilde \omega_h\| \leq \int_{\Gamma} N_\gamma \dif \bs \omega(\gamma)  \leq \int_{\Gamma} \mathbf 1_{\{|\gamma^x(\bar t)- x| \leq K\}}(\gamma)\left(1+\frac{3}{|J|}\right) \mr{Tot.Var.}\gamma^v \dif \bs \omega(\gamma) < \infty
$$
the last inequality being a consequence of \eqref{eq:CKTlr}.
By an elementary transversality argument, for $\tilde \omega_h$-a.e. curve $(\gamma^\prime, I_{\gamma^\prime})$ it holds, since 
$$
|\dot \gamma^{\prime, x}(t) - f^\prime(\bar v)| > \mathfrak h^-(\bar v, \delta) \quad \text{for a.e. $t \in (t_{j}^{\gamma,-}, t_{j}^{\gamma,+})$}
$$
that
\begin{equation}\label{eq:transvers}
\msc L^1 \Big\{ t \in I_{\gamma^\prime} \; \Big| \; \gamma^\prime(t) \in S_{\varepsilon,r} \times \bar J \Big\} \leq  \frac{2\varepsilon}{\mathfrak h^-(\bar v, \delta)}.
\end{equation}
By construction, it holds
\begin{equation}\label{eq:tildeomegain}
\int_{\Gamma} (\mr{id}_t, \gamma)_{\sharp} \msc L^1\llcorner I_{\gamma} \dif \tilde \omega_h \geq \msc L^3 \llcorner \Big\{(t, x, v) \in S^\gamma_{\varepsilon, r} \times \bar J \; \Big| \;  u(t,x) > v      \Big\}.
\end{equation}
The measure of the set in the right hand side of \eqref{eq:tildeomegain} is at least $r\varepsilon |\bar J|$, therefore combining \eqref{eq:transvers}, \eqref{eq:tildeomegain}, we obtain
\begin{equation}
\tilde \omega_h (\Gamma) \geq  \frac{1}{2}r |\bar J| \mathfrak h^-(\bar v, \delta).
\end{equation}
We let $\Gamma = \Gamma_1 \cup \Gamma_2 \cup \Gamma_3$, where 
$$
\begin{aligned}
& \Gamma_1 = \big\{ (\gamma, I_{\gamma}) \; \big| \; |I_{\gamma} | \geq r\big\},\\
& \\
& \Gamma_2 = \big\{ (\gamma, I_{\gamma}) \; \big| \; |I_{\gamma}| < r, \quad \gamma^v(\partial I_{\gamma}) \cap  \partial J\neq 0 \big\},\\
& \\
& \Gamma_3 = \big\{ (\gamma, I_{\gamma}) \; \big| \; |I_{\gamma}| < r, \quad \gamma^v(\partial I_{\gamma})\cap  \partial J = 0 \big\}.
\end{aligned}
$$
For $\tilde \omega_h$-a.e. $(\gamma, I_{\gamma}) \in \Gamma_1$ it holds
$$
\msc L^1 \big\{ t \in I_{\gamma} \; \big| \; (t, \gamma(t)) \in B_{2r}(\bar t, \bar x) \times J\big\} \geq r 
$$
For $\tilde \omega_h$-a.e. $(\gamma, I_{\gamma}) \in \Gamma_2$ it holds
$$
\gamma(I_{\gamma}) \subset B_{2r}(\bar t, \bar x)\times J, \qquad \mr{Tot.Var.} \gamma^v > |J|/3
$$
For $\tilde \omega_h$-a.e. $(\gamma, I_{\gamma}) \in \Gamma_3$ it holds
$$
(t^1_{\gamma}, \gamma(t^1_{\gamma})), \quad (t^2_{\gamma}, \gamma(t^2_{\gamma})) \in  B_{2r}(\bar t, \bar x) \times J
$$
Then it follows that one of these condition holds: 
\begin{equation}
\begin{aligned}
& \tilde \omega_h(\Gamma_1) \geq \frac{r|\bar J| \mathfrak h^-(\bar v, \delta)}{3}, \quad \tilde \omega_h(\Gamma_2) \geq \frac{r|\bar J| \mathfrak h^-(\bar v, \delta)}{3}, \quad \tilde \omega_h(\Gamma_3) \geq \frac{r|\bar J| \mathfrak h^-(\bar v, \delta)}{3}.
\end{aligned}
\end{equation}
If the first condition holds, we deduce from $(e_t)_\sharp \tilde \omega_h \leq (e_t)_\sharp  \omega_h = \chi(t, \cdot, \cdot) \cdot \mathscr L^2$ for every $t> 0$, using Fubini's theorem,
$$
\msc L^2\Big( \big\{    (t, x) \in B_{2r}(\bar t, \bar x) \; \big| \;  u(t, x) > \bar v- \delta \big\} \Big) \geq \frac{1}{3}r^2 \mathfrak h^-(\bar v, \delta).
$$
If the second condition holds,  we deduce, using the first equation in \eqref{eq:inducedpair}, that
$$
|\mu_1|(B_{2r}(\bar t, \bar x) \times J ) \geq \int_{\Gamma_2} \frac{|J|}{3}\dif \tilde \omega_h(\gamma) \geq \frac{r|\bar J| \mathfrak h^-(\bar v, \delta)}{3}\frac{|J|}{3} \geq \frac{1}{27 \Vert f^{\second} \Vert_{\infty}^2} r \mathfrak h^-(\bar v, \delta)^3.
$$
If the third condition holds, in a similar way, we deduce, using the second equation in \eqref{eq:inducedpair}, that 
$$
|\mu_0|(B_{2r}(\bar t, \bar x) \times J ) > \frac{1}{9\Vert f^{\second} \Vert} r \mathfrak h^-(\bar v, \delta)^2.
$$
This proves the result.
\end{proof}

The symmetric statement holds for the epigraph: the proof is identical, therefore is omitted. 
\begin{lemma}\label{lemma:epireg}
    Let $(\gamma, I_{\gamma}) \in \Gamma_e$,  let $$\bar v = \min\, [\gamma^v(\bar t\pm)],
$$
    where 
    $$
    [\gamma^v(\bar t\pm)]=
    \begin{cases}
      \{\gamma^v(\bar t+),\, \gamma^v(\bar t-)\}
        \  &\text{if}\quad  \bar t \in I_{\gamma},
        \\
        \{\gamma^v(\bar t+)\}
        \   &\text{if}\quad  \bar t = \inf I_\gamma,
        \\
        \{\gamma^v(\bar t-)\}
        \   &\text{if}\quad  \bar t = \sup I_\gamma.
    \end{cases}
    $$
    and set $\bar x = \gamma^x(\bar t)$. Let $$\bar v =\liminf_{t \to \bar t, \; t \in I_\gamma} \gamma^v(t).$$ Then there exists an absolute constant $c$ depending only on $\Vert f^{\second}\Vert_{\infty}$ such that  for every $\delta \in (0, 1)$ at least one of the following holds true.
    \begin{equation}
    \begin{aligned}
       &  \liminf _{r \downarrow 0} \frac{\msc L^2 \big\{(t, x) \in B_R(\bar t, \bar x) \; \big| \; u(t,x) < \bar v-\delta\big\}}{r^2} > c \cdot \mathfrak h^+(\bar v, \delta)\\
   & \limsup_{r \downarrow 0} \frac{\nu_0(B_R(\bar t, \bar x))}{r} > c \cdot \mathfrak h^+(\bar v, \delta)^2\\
   & \limsup_{r \downarrow 0} \frac{\nu_1(B_R(\bar t, \bar x))}{r} > c \cdot \mathfrak h^+(\bar v, \delta)^3.
        \end{aligned}
\end{equation}
\end{lemma}
The following proposition concludes the proof of Step 4.
\begin{prop}\label{prop:notvmo}
For $\nu^+_{1, J}$-a.e. $(t, x) \in (0, T) \times \mathbb R$, it holds
\begin{equation}
    \limsup_{r \downarrow 0} \frac{\nu(B_r(t, x))}{r} > 0.
\end{equation}
    In particular, the measure $\nu^+_{1,J}$ is concentrated on the 1-rectifiable set $\mathbf J$ of Theorem \ref{thm:dimdthm}. 
\end{prop}
\begin{proof}
    For $\nu^+_{1,J}$ a.e. $(t, x)$ one of the following holds:
    \begin{enumerate}
        \item there exists $(\gamma, t, x, v) \in \mc G^+_{h, J}$ and $(\gamma^{\prime}, t^{\prime}, x^{\prime}, v^{\prime}) \in \mc G^-_e$ such that 
        $$
        (t, x) = (t^{\prime}, x^{\prime}) \quad \text{and} \quad \gamma^v(t-)> v = v^{\prime} \geq \gamma^{\prime,v}(t-);
        $$
         \item there exists $(\gamma, t, x, v) \in \mc G^-_{h}$ and $(\gamma^{\prime}, t^{\prime}, x^{\prime}, v^{\prime}) \in \mc G^-_{e,J}$ such that 
        $$
        (t, x) = (t^{\prime}, x^{\prime}) \quad \text{and} \quad \gamma^{\prime, v}(t-)< v = v^{\prime} \leq \gamma^{v}(t-).
        $$
    \end{enumerate}
Since the proof is symmetrical, assume the first condition holds. We apply Lemma \ref{lemma:hypreg} to the curve $\gamma$ and Lemma \ref{lemma:epireg} to the curve $\gamma^{\prime}$ with 
$$
\delta = \frac{1}{3}\{|\gamma^v(t-)-v| \}.
$$
Then either the second or the third condition holds in at least one of the two Lemma \ref{lemma:hypreg} , \ref{lemma:epireg}, or the first condition holds in both Lemmas.  But in this case, $(t, x)$ cannot be a point of vanishing mean oscillation. Therefore by Theorem \ref{thm:dimdthm}, it must holds $(t, x) \in \mathbf J$.
\end{proof}

\vspace{0.5cm}
\textbf{5.} From the previous steps, we conclude that $\nu^+_1$ is concentrated on a 1-rectifiable set. In fact, one has 
\begin{equation}
    \nu^+_1 = (p_{t,x})_{\sharp} \int_{\Gamma_h} \mu_1^{\gamma,+} \dif \bs \omega_h(\gamma) \leq  (P_{t,x})_{\sharp} \pi^+ = \nu^+_{1, J} + \sum_{l \in \mathbb N} \nu^+_{1,l}\,.
\end{equation}
From Step 3 and Step 4 we deduce that $\nu^+_1 $ is concentrated on a 1-rectifiable set.
The same argument holds for $\nu_1^-$, therefore the proof of Theorem \ref{thm:1drect} is completed. 

\vspace{0.5cm}
We conclude the section with a result about the structure of the source term $\mu_0$ outside the jump set $\mathbf J$ that, beside having an interest on its own, it will be useful for later when studying the isentropic system with $\gamma = 3$.
\begin{prop}\label{prop:sourcestruct}
    In the above setting, there exists a measurable function $e(t,x)$ such that one additionally has 
    \begin{equation}\label{eq:mu0form}
        \mu_0 \llcorner \mathbf J^c = e(t,x) (p_{t,x})_\sharp |\mu_0| \otimes \delta_{\bar u(t,x)}, \qquad e(t,x) \in \{1, -1\}.
    \end{equation}
\end{prop}

\begin{proof}
    Since $\mu_1, \mu_0$ is a minimal pair, it is induced by a \emph{good} Lagrangian representation $\omega_h$ (Definition \ref{defi:goodlr}) in particular,
    $$
    \mu_0 =  \int_{\Gamma} \mu_0^\gamma \dif \omega_h(\gamma)
    $$
    Since $(\mu_0, \mu_1)$ and $(-\mu_0, -\mu_1)$ are simultaneously induced by $\omega_h, \omega_e$, there holds
    $$
    -\mu_0 = \int_{\Gamma} \mu_0^\gamma \dif \omega_e(\gamma)
    $$
    This means that for $|\mu_0|$ almost every point $(t,x,v)$ there is a curve $\gamma_h \in \Gamma_h$ and a curve $\gamma_e$ in $\Gamma_e$ such that $\gamma_h(t) = (x, v), \gamma_e(t) = (x,v)$, and either 
    $$
    t = t^1_{\gamma_h}, \quad t = t^2_{\gamma_e}
    $$
    or 
    $$
    t = t^2_{\gamma_h}, \quad t = t^1_{\gamma_e}.
    $$
Assume now that $\mu_0$ is not of the form \eqref{eq:mu0form}. Then there is a point $(t,x) \in \mathbf J^c$, two values $v_- < v_+$ and two curves $\gamma_h \in \Gamma_h, \gamma_e \in \Gamma_e$ such that $\gamma_h^x(t) =\gamma_e^x(t) = x$ and
$$
\limsup_{s \to t, \; s \in I_{\gamma_h}} \gamma^v_h(s) = v_+, \qquad \limsup_{s \to t, \; s \in I_{\gamma_e}} = v_-
$$
Since $(t,x) \in \mathbf J^c$, the first option of Lemma \ref{lemma:hypreg} and of Lemma \ref{lemma:epireg} must hold. But with the same argument of Proposition \ref{prop:notvmo}, this is a contradiction because $(t,x)$ is of vanishing mean oscillation.    
\end{proof}

\section{Regularity of Burgers' Equation}\label{sec:RBE}
In this section we provide a first application of the Lagrangian representation to obtain a regularity result for quasi-entropy solutions Burger's equation.
We first recall the Definition of the Besov spaces $B^{\alpha, p}_{\infty, \mr{loc}}(\mathbb R)$.
\begin{defi}
   Let $\alpha \in (0, 1)$, $p \in [1, +\infty)$. A function $u: \mathbb R \to \mathbb R$ belongs to $B^{\alpha, p}_{\infty, \mr{loc}}(\mathbb R)$ if 
    \begin{equation}
       \Vert u\Vert^p_{B^{\alpha, p}_{\infty}(K)}:= \sup_{h > 0} \int_{K} \left| \frac{u(x+h)-u(x)}{h^{\alpha}}\right|^p \dif x  < +\infty, \quad \forall \; K \subset \mathbb R\quad \text{compact}.
    \end{equation}
\end{defi}
Then we have the following regularity result in terms of Besov spaces.
\begin{thm}\label{thm:besov}
     Let $u$ be a quasi-entropy solution to Burgers equation such that in addition $\mu_1$ given by Proposition \ref{prop:feskin} is a signed measure.
Then for every $T> 0$, $K \subset \mathbb R$ compact and $\delta > 0$ there exists a constant $C \equiv C(K, T, \mu_0)$ such that 
\begin{equation}
\int_{\delta}^T \Vert u(t)\Vert_{B^{1/2, 1}_{\infty}(K)} \dif t < \frac{C}{\min\{\delta,1\}}.
\end{equation}
 \end{thm}

\begin{proof}
Up to a symmetry in the $v$, we can assume that $\mu_1$ is a positive measure.

  Moreover, up to choosing a different pair $(\hat \mu_0, \hat \mu_1)$ smaller than $(\mu_0, \mu_1)$ for $\preceq$, we can assume that $\bs \omega_h$ is a Lagrangian representation of the hypograph of $u$ (Definition~\ref{def:lagrangianrep}) that induces $(\mu_0, \mu_1)$ as in Definition~\ref{defi:goodlr}. Thanks to Proposition~\ref{prop:goodpairs}, we can also assume that $(\mu_0, \mu_1)$ is simultaneously induced (Definition \ref{defi:siminduc}) by $\bs \omega_h$ and $\bs \omega_e$, where $\bs \omega_e$ is a Lagrangian representation of the epigraph of $u$ (Definition~\ref{defi:lagraepi}). This second step is not really necessary but makes the proof easier. A key point is that these operation preserve the sign of $\mu_1$, by definition of the relation $\preceq$. Finally, we denote by $\Gamma_h$, $\Gamma_e$ the set of curves selected by Lemma~\ref{lemma:goodcurves}.
  
\vspace{0.5cm}
\textbf{1.} Fix $\Delta t > 0$ and for $t \geq \Delta t$ consider the set of curves
$$
\Gamma_h^{t, \Delta t} := \Big\{ \gamma_h \in \Gamma_h \quad \big| \quad t_{\gamma_h}^- \leq t-\Delta t, \quad t^+_\gamma \geq t \Big\}
$$
$$
\Gamma_e^{t, \Delta t} := \Big\{ \gamma_e \in \Gamma_e \quad \big| \quad t_{\gamma_e}^- \leq t-\Delta t, \quad t^+_\gamma \geq t \Big\}
$$
Define the measures 
\begin{equation}
    \begin{aligned}
        \chi^t_{a^{\Delta t}} := {e_t}_{\sharp} \;  (\bs \omega_h \llcorner \Gamma_h^{t, \Delta t}) \leq \chi_h(t, \cdot, \cdot )  \msc L^2 \llcorner ( \mathbb R \times (0, 1)) \quad \text{in $\msc M(\mathbb R\times (0, 1))$},& \\
        & \\
         \chi^t_{b^{\Delta t}} := {e_t}_{\sharp} \;  (\bs \omega_e \llcorner \Gamma_e^{t, \Delta t}) \leq \chi_e(t, \cdot, \cdot )  \msc L^2 \llcorner ( \mathbb R \times (0, 1)) \quad \text{in $\msc M(\mathbb R\times (0, 1))$}.
    \end{aligned}
\end{equation}
Finally, we define the functions
\begin{equation}
    \begin{aligned}
       &  a^{\Delta t}(t, x) : = \mr{sup} \; \Big\{ v \in (0, 1) \; \big| \;  (v, x) \in \mr{supp} \; \chi^t_{a^{\Delta t}} \Big\}, \qquad \forall \; x \in \mathbb R\\
       & \\
       &  b^{\Delta t}(t, x) : = \mr{inf} \; \Big\{ v\in (0, 1) \; \big| \;  (v, x) \in \mr{supp} \; \chi^t_{b^{\Delta t}} \Big\}, \qquad \forall \; x \in \mathbb R.
    \end{aligned}
\end{equation}
Notice that the $L^1$ distance between $a^{\Delta t}(t, \cdot)$ and $u(t, \cdot)$ can be estimated in terms of the source $\mu_0$. In fact, by definition of $a^{\Delta t}$, we have
\begin{equation}
\begin{aligned}
    \int_{-M}^M |a^{\Delta t}(t, x) -u(t, x) | \dif x & = \bs \omega_h\Big(\Big\{ \gamma_h \in \Gamma_h \quad \big| \quad t \in I_{\gamma_h}, \quad t^-_{\gamma_h} \in (t-\Delta t, t)\Big\}\Big)\\
    & \\
    & \leq \mu_0^+(S_t^{M, \Delta t})\\
    \end{aligned}
\end{equation}
where 
$$
S_t^{M, \Delta t}:= (t-\Delta t, t) \times (-M- \Delta t, M + \Delta t) \times (0, 1).
$$
This implies, integrating in $(0, T)$ and using Fubini's Theorem:
\begin{equation}\label{eq:uminusa}
    \int_{\Delta t}^T \int_{-M}^M |a^{\Delta t}(t, x) -u(t, x) | \dif x \dif t \leq \Delta t \cdot |\mu_0^+|\big([0,T] \times (-M-\Delta t, M+\Delta t) \times (0, 1) \big).
\end{equation}
Entirely symmetrical statements holds for the distance of $u$ from $b^{\Delta t}$:
\begin{equation}\label{eq:uminusb}
\begin{aligned}
    &  \int_{-M}^M |a^{\Delta t}(t, x) -u(t, x) | \dif x \leq \mu_0^-(S_t^{M, \Delta t}),\\
    & \int_{\Delta t}^T \int_{-M}^M |b^{\Delta t}(t, x) -u(t, x) | \dif x \dif t \leq \Delta t \cdot |\mu_0^-|\big([0,T] \times (-M-\Delta t, M+\Delta t) \times (0, 1) \big).\\
    \end{aligned}
\end{equation}
By triangular inequality, the difference $a^{\Delta t}(t, x)-b^{\Delta t}(t, x)$ lies in $L^1$ as well and 
\begin{equation}\label{eq:abdifference}
\begin{aligned}
& \int_{-M}^M |a^{\Delta t}(t, x)-b^{\Delta t}(t, x)| \dif x  \leq  |\mu_0|(S_t^{M, \Delta t}),\\
    & \int_{\Delta t}^T \int_{-M}^M |a^{\Delta t}(t, x)-b^{\Delta t}(t, x)| \dif x \dif t \leq \Delta t |\mu_0|\big([0,T] \times (-M-\Delta t, M+\Delta t) \times (0, 1) \big).\\
    \end{aligned}
\end{equation}

\vspace{0.5cm}
\textbf{2.}
Now fix $x < y$ and $\varepsilon> 0$ small. Take a curve $\bar \gamma_h \in \Gamma_h^{t, \Delta t}$ such that $|\bar \gamma_h(t) -(y, a^{\Delta t}(t,y))| \leq \varepsilon$. Since $\mu_1$ is positive and $(y, a^{\Delta t}(t,y)) \in \mr{supp} \; \chi_{a^{\Delta t}}^t$, by~\eqref{eq:mu1def} we can assume that $t \mapsto \bar \gamma_h^v(t)$, $t \in I_{\bar \gamma_h}$, is decreasing.
By definition of $b^{\Delta t}$ there exist positive measure set of curves $G \subset \Gamma_e^{t, \Delta t}$ such that  $$|\gamma_e(t)-(x, b^{\Delta t}(t,x))| \leq \varepsilon.$$
By Lemma~\ref{lemma:nocrossing}, it holds
\begin{equation}\label{eq:nocrossingG}
\omega_e\Big( \Big\{ \gamma_e \in G \quad \big| \quad \gamma_e(t-\Delta t) >\bar \gamma_h(t-\Delta t) \quad \text{and} \quad \gamma_e(t) < \bar \gamma_h(t)\Big\}\Big) = 0
\end{equation}
Moreover, since $\mu_1 \geq 0$ and again by \eqref{eq:mu1def}, $I_{\gamma_e} \ni t \mapsto \gamma^v_e(t)$ is increasing for $\omega_e$-a.e. $\gamma_e \in G$. Therefore, thanks to the characteristic equation~\eqref{eq:charaeqla}, it holds
\begin{equation}\label{eq:Oleinikpreineq}
\begin{aligned}
    &\bar \gamma_h^x(t-\Delta t) < y -\Delta t\cdot a^{\Delta t}(t, y) + 2\varepsilon\\
    & \gamma_e^x(t-\Delta t) > x -\Delta t\cdot b^{\Delta t}(t, x) - 2\varepsilon, \quad \text{for $\omega_e$-a.e. $\gamma_e \in G$}.
    \end{aligned}
\end{equation}
Then, combining~\eqref{eq:nocrossingG} with~\eqref{eq:Oleinikpreineq} and letting $\varepsilon \to 0$ we obtain
\begin{equation}
    a^{\Delta t}(t, y) -b^{\Delta t}(t, x) \leq \frac{y-x}{\Delta t} , \qquad \text{for every $x < y$}
\end{equation}
Setting $h = y-x$ and integrating in an interval $[-M, M]$, for some $M > 0$, we obtain
\begin{equation}
    \int_{-M}^M (a^{\Delta t}(t, x+h)-b^{\Delta t}(t, x) )^+ \dif x \leq 2 M h \Delta t^{-1}.
\end{equation}
Moreover, combining with the inequality above the $L^{\infty}$ bounds for $a^{\Delta t}, b^{\Delta t}$,  and the inequality~\eqref{eq:abdifference} of Step 1, we obtain
\begin{equation}
    \begin{aligned}
        -\int_{-M}^M & (a^{\Delta t}(t, x+h)-b^{\Delta t}(t, x))^-  \dif x  = -\int_{-M}^M (a^{\Delta t}(t, x+h)- b^{\Delta t}(t, x))^+ \dif x \\
        &+ \int_{-M}^M (a^{\Delta t}(t, x+h)- b^{\Delta t}(t, x)) \dif x\\
        & \geq -2M h\Delta t^{-1} -2h-|\mu_0|(S_t^{M, \Delta t}) \\
    \end{aligned}
\end{equation}
which in turn yields the bound 
\begin{equation}\label{eq:Oleinik}
\begin{aligned}
    \int_{-M}^M |a^{\Delta t}(t, x+h) &-b^{\Delta t}(t, x)|  \dif x \\
    & \leq C \Big( h\Delta t^{-1}+ |\mu_0|(S_t^{M, \Delta t})\Big),
\end{aligned}
\end{equation}
for all $t > 0$ and $\Delta t < t$, with $C$ depending only on $M$. 

\vspace{0.5cm}
\textbf{Step 3.} Fix $\delta > 0$ and  $h^{1/2} < \delta$. Then, for every $\Delta t < \delta$, we estimate the $L^1$ norm of the difference $\bs u(t,x) -u(t,x+h)$ by
\begin{equation}
    \begin{aligned}
        \int_{\delta}^T \int_{-M}^M |u(t, x) -u(t, x+h)| \dif x  \dif t & \leq  \int_{\delta}^T \int_{-M}^M |u(t, x) -b^{\Delta t}(t, x)| \dif x \dif t+\\
        & + \int_{\delta}^T \int_{-M}^M |b^{\Delta t}(t, x)-a^{\Delta t}(t, x+h)| \dif x \dif t + \\
     & + \int_{\delta}^T\int_{-M}^M |a^{\Delta t}(t, x+h)-u(t, x+h)| \dif x \dif t \\
        & \leq 2 \Delta t\cdot |\mu_0|((0, T) \times (-M-\delta, M +\delta) \times (0, 1)) \\
        & + \int_{\delta}^T \int_{-M}^M |a^{\Delta t}(t, x+h)-b^{\Delta t}(t, x)| \dif x \dif t
    \end{aligned}
\end{equation}
where the inequality in the last line follows by~\eqref{eq:uminusa}, \eqref{eq:uminusb}.
By \eqref{eq:Oleinik}, since for $t \in (\delta, T)$ one has $\Delta t < \delta < t$, we obtain
$$
 \int_{\delta}^T \int _{\mathbb R} |a^{\Delta t}(t, x+h)-b^{\Delta t}(t, x)| \dif x \dif t \leq C(h\Delta t^{-1} + \Delta t)
$$
so that in total, for another constant $C$ depending on $M, T$ and $|\mu_0|$, we obtain
\begin{equation}
    \int_{\delta}^T \int_{-M}^M |u(t, x) -u(t, x+h)| \dif x  \dif t  \leq C(h\Delta t^{-1} + \Delta t), \qquad \forall \; \Delta t < \delta
\end{equation}
Choosing $\Delta t = h^{1/2}$ (which is possible since $h^{1/2} < \delta$ by assumption), we obtain 
\begin{equation}
 \int_{\delta}^T \int_{-M}^M \frac{|u(t, x) -u(t, x+h)|}{h^{1/2}} \dif x  \dif t  <2 C, \qquad \forall h \quad \text{such that} \; h^{1/2} < \delta
\end{equation}
Instead if $h^{1/2} \geq \delta$, we obtain trivially
\begin{equation}
 \int_{\delta}^T \int_{-M}^M \frac{|u(t, x) -u(t, x+h)|}{h^{1/2}} \dif x  \dif t  \leq\frac{2MT}{\delta}.
\end{equation}
This proves that
\begin{equation}
 \sup_{h > 0} \int_{\delta}^T \int_{K} \frac{|u(t, x) -u(t, x+h)|}{h^{1/2}} \dif x  \dif t  \leq \frac{C}{\delta}, \qquad \forall \quad K \subset \mathbb R \; \text{compact}.
\end{equation}
\end{proof}

\section{Applications to the Euler system with $\gamma = 3$}\label{sec:isentropic}
In this section we apply the results of this paper to the system of isentropic gas dynamics with $\gamma = 3$, for the evolution of the density $\rho$ and the momentum $m = \rho u$ of an isentropic gas:
\begin{equation}\label{eq:iso}
\begin{aligned}
    & \partial_t \rho + \partial_x m = 0, \\
    & \partial_t m  + \partial_x (m^2/\rho + \rho^3/3) = 0
    \end{aligned}
\end{equation}
The Riemann invariants of the system are 
\begin{equation}\label{eq:RIiso}
w = \frac{m}{\rho} -\rho, \qquad  z = \frac{m}{\rho} + \rho
\end{equation}
with corresponding eigenvalues
$$
\lambda_1(w,z) = w, \qquad \lambda_2(w,z) = z.
$$
We consider solutions that take values in the compact set
$$
\mathcal K := \Big\{(\rho, m) \; \Big| \; |(\rho, m)| \leq M, \quad \rho \geq c\Big\}
$$
where $c, M$ are fixed constants from now on.
\begin{defi}[Entropy solutions in $\mc K$]\label{defi:entropyK}
A function $\bs u = (\rho, m) \in \mathbf L^\infty([0,T] \times \mathbb R; \mc K)$ is an entropy solution in $\mc K$ of \eqref{eq:iso} if it is a weak solutions of \eqref{eq:iso} and it dissipates every convex entropy $\eta : \mathcal K \to \mathbb R$:
$$
\partial_t \eta(\bs u) + \partial_x q(\bs u) \leq 0 \quad \text{in $\mathscr D^\prime$}
$$
where $q$ is the corresponding entropy flux $\nabla q =\nabla \eta  D f$, and 
$$
f = \begin{pmatrix}
    m \\
    \frac{\rho^3}{3} + \frac{m^2}{\rho}
\end{pmatrix}.
$$
\end{defi}
\begin{remark}
    Some comments are here in order. First,  recall that the convexity of the entropy must be checked in the conserved variables $(\rho, m)$. Secondly, a more subtle point is the following: in the definition of an entropy solution in $\mc K$ we require \emph{all} convex entropies $\eta : \mathcal{K} \to \mathbb{R}$ to be dissipated. However, it may happen that such entropies cannot be extended as convex functions to the whole set $\{(\rho, m) \in \mathbb R^+ \times \mathbb R) \; | \; \rho \neq 0\}$. An example of class of solutions that fits our definition is the one vanishing viscosity solutions with the identity  viscosity matrix such that $\mr{Im} \, \bs u^\varepsilon \subset \mathcal K$ for all  $\varepsilon > 0$, where $\bs u^\varepsilon$ are the viscous approximations solving 
    $$
    \bs u^\varepsilon_t + f(\bs u^\varepsilon)_x = \varepsilon \bs u^\varepsilon_{xx}.
    $$
\end{remark}
We recall that a function $\eta\in C^2$ is an entropy of the system if and only if it satisfies, in Riemann coordinates $(w,z)$,
\begin{equation}\label{eq:entropyiso}
    \eta_{wz}(w,z) = \frac{-\lambda_{1z}}{\lambda_1-\lambda_2} \eta_w(w,z) +  \frac{\lambda_{2w}}{\lambda_1-\lambda_2} \eta_z(w,z) = 0. 
\end{equation}
It is well known that the system \eqref{eq:iso} admits a strictly convex entropy, the energy:
$$
\eta_E(\rho, m) = \frac{1}{2} m^2/\rho + \frac{1}{6}\rho^3.
$$
In the following Proposition we use the energy to prove that, if an entropy solution $(\rho, m)$ satisfies $\rho \geq c >0$ in $[0,T] \times \mathbb R$, then its Riemann invariants are quasi-entropy solutions to the Burgers equation.
\begin{prop}\label{prop:wzfes}
    Let $\bs u = (\rho, m) \in \mathbf L^\infty([0,T] \times \mathbb R, \mc K)$ be an entropy solution in $\mc K$ of \eqref{eq:iso}. Then $w = u-\rho$ and $z = u+ \rho$ are quasi-entropy solutions of the Burgers equation, i.e. they satisfy Definition \ref{defi:fes} with $f (u)= u^2/2$.
\end{prop}

\begin{proof}
    We prove that $w$ is a quasi-entropy solution of the Burgers equation, the proof for $z$ being identical. By definition of quasi-entropy solution, we need to show that for every $\eta : \mathbb R\to \mathbb R$, $\eta \in C^2$, there holds
    $$
    \mu_\eta = \partial_t \eta(w) + \partial_x q(w) \in \mathscr M([0,T] \times \mathbb R)
    $$
    where $q$ is the entropy flux of $\eta$, $q^\prime(w) = w \eta^\prime(w)$. Starting from $\eta$, we construct an entropy-entropy flux pair for the isentropic system \eqref{eq:iso} by setting, in Riemann invariants,
    $$
    \bar \eta(w,z) := \eta(w), \qquad \bar q(w, z) := q(w) \qquad \forall \; w < z.
    $$
It is immediate to check that, if $\eta, q$ are entropy entropy fluxes for the Burgers equation, then $\bar \eta, \bar q$ are entropy - entropy fluxes for the isentropic system \eqref{eq:iso}. We now show that $\eta(\rho, m)$ is $C^2$ in the region $\rho >0$ when considered as a function of the original variables $\rho, m$. This is fairly trivial, because we have $\eta \in C^2$, and morevoer the change of coordinates \eqref{eq:RIiso} is $C^\infty$ in $\rho > 0$. This shows that $\bar \eta \in C^2(\mathbb R^+\times \mathbb R \setminus \{\rho = 0\})$.
Clearly, since the energy $\eta_E$ is strictly convex, $\bar \eta$ is $C^2$ and $\mc K$ is a compact set, there exists a constant $C > 0$ big enough such that the entropy $S := \bar \eta + C \eta_E$ is convex in $\mc K$. Therefore, since by assumption $\bs u$ is an entropy solution in $\mc K$ of \eqref{eq:iso} (Definition \ref{defi:entropyK}), we have
\begin{equation}
    0 \geq \mu_S := \partial_t S(\bs u) + \partial_x Q(\bs u) \in \msc M([0,T] \times \mathbb R)
\end{equation}
where $Q$ is the entropy flux of $S$, therefore $\mu_S$ is a nonpositive measure. But now by linearity of the operator $\eta \mapsto \mu_{\eta}$ we can write 
$$
\mu_\eta = \mu_{\bar \eta} = \mu_S - C \mu_{\eta_E} \in \msc M([0,T] \times \mathbb R)
$$
and this proves the claim.
\end{proof}

The following Theorem is due to Lions, Perthame and Tadmor in \cite{LPT_kineticise}. We consider here directly the version for the exponent $\gamma = 3$.
\begin{thm}\label{thm:lptkin}
    Let $\bs u$ be an entropy solution to \eqref{eq:iso}. Then for some nonpositive measure $m \in \msc M([0,T] \times \mathbb R)$, there holds
    \begin{equation}
        \partial_t g(t,x,v)  + v \partial_x g(t,x,v) = \partial_{vv} m
    \end{equation}
    where 
    $$
    g(t,x,v) := \begin{cases}
        1 & \text{if $w(t,x) \leq v \leq z(t,x)$}\\
        0 & \text{otherwise}
    \end{cases}
    $$
    where $w(t,x), z(t,x)$ are the Riemann coordinates of $\bs u(t,x)$.
\end{thm}

In the following theorem we use Theorem \ref{thm:1drect} and Proposition \ref{prop:wzfes} to deduce the measure $(p_{t,x})_\sharp m$ is concentrated on a $1$-rectifiable set $\mathbf S \subset [0,T] \times \mathbb R$.
\begin{thm}\label{thm:isereg}
   Let $\bs u =(\rho, m) \in \mathbf L^\infty([0,T] \times \mathbb R; \mathcal K)$ be an entropy solution of \eqref{eq:iso}. Then $(p_{t,x})_\sharp m$ is concentrated on a $1$-rectifiable set $\mathbf S \subset [0,T] \times \mathbb R$.
\end{thm}
\begin{proof}
   \textbf{1.} From Proposition \ref{prop:wzfes} we deduce that $w, z$ are quasi-entropy solutions to the Burgers equation. By Proposition \ref{prop:feskin} this is equivalent to 
    \begin{equation}\label{eq:zchi}
        \partial_t \chi  + v \partial_x \chi = \mu_0 + \partial_v \mu_1, \qquad \chi(t,x,v) := \begin{cases}
            1 & \text{if $ v \leq z(t,x)$}\\
            0 & \text{otherwise}
        \end{cases}
    \end{equation}
    \begin{equation}\label{eq:wchi}
        \partial_t \psi + v \partial_x \psi = \sigma_0 + \partial_v \sigma_1, \qquad \psi(t,x,v) := \begin{cases}
            1 & \text{if $v \leq w(t,x)$}\\
            0 & \text{otherwise}
            \end{cases}
    \end{equation}
for some locally finite measures $\mu_0, \sigma_i$.

Recall the Definition of the sets $\mathbf P_\chi, \mathbf P_\psi$ given in \eqref{eq:pairset}, which are the sets of all the pairs $(\mu_0, \mu_1)$ and $(\sigma_0, \sigma_1)$ of locally finite measures for which \eqref{eq:zchi} and \eqref{eq:wchi} are satisfied, respectively. We can assume, up to replacing the measures $\mu_i, \sigma_i$, that the pairs $(\mu_0, \mu_1)$ and $(\sigma_0, \sigma_1)$ are minimal with respect to the relation  $\preceq$ introduced in Definition \ref{defi:relation}. Therefore from Theorems \ref{thm:dimdthm}, \ref{thm:1drect} we readily deduce that $\nu_1 := (p_{t,x})_\sharp |\mu_1|$ and $\xi_1 := (p_{t,x})_\sharp |\sigma_1|$ are concentrated on the $1$-rectifiable set $\mathbf S \subset [0,T] \times \mathbb R$ which is defined as 
\begin{equation}
    \mathbf S := \Bigg\{(t,x) \in [0,T] \times \mathbb R \; \Big| \; \limsup_{r \downarrow 0} \frac{(p_{t,x})_\sharp\big(|\mu_1| +|\sigma_1| + |\mu_0| + |\sigma_0|\big)(B_r(t,x))}{r} > 0  \Bigg\}
\end{equation}

    \vspace{0.3cm}
     \noindent \textbf{2.}  Consider the disintegrations, for $i= 0,1$,
     $$
     \mu_i = \mu_{i}^{(t,x)} \otimes \nu_i, \qquad \sigma_i = \sigma_{i}^{(t,x)} \otimes \xi_i$$
     $$\nu_i := (p_{t,x})_\sharp |\mu_i|, \quad \xi_i := (p_{t,x})_\sharp |\sigma_i|
     $$
Thanks to Proposition \ref{prop:sourcestruct}, we deduce that for some measurable functions $\bar w(t,x)$ and $\bar z(t,x)$
\begin{equation}\label{eq:inSc}
\begin{aligned}
& \mu_0^{(t,x)} = e(t,x)\delta_{\bar z(t,x)},\qquad \text{for  $\nu_0$-a.e. $(t,x) \in \mathbf S^c$} \\
& \sigma_0^{(t,x)} = \wt e(t,x)\delta_{\bar w(t,x)},\qquad \text{for  $\xi_0$-a.e. $(t,x) \in \mathbf S^c$}
\end{aligned}
\end{equation}
where $e, \wt e$ are some measurable functions with values only in $\{-1,1\}$.
Notice that 
$$
\rho(t,x) = z(t,x)- w(t,x) = \int_{\mathbb R} \chi(t,x,v) - \psi(t,x,v) \dif v, 
$$
$$
m(t,x) = \int_{w(t,x)}^{z(t,x)} v \dif v =  \int_{\mathbb R} v (\chi(t,x,v)-\psi(t,x,v)) \dif v
$$
therefore by subtracting \eqref{eq:wchi} from \eqref{eq:zchi} and integrating in $v$ we obtain the equation for the conservation of mass
\begin{equation}
    0 = \partial_t \rho + \partial_x m = (p_{t,x})_\sharp ( \mu_0 - \sigma_0)  \qquad \text{in $\mathscr D^\prime_{t,x}$}.
\end{equation}
Using \eqref{eq:inSc}, we therefore deduce that 
$$
e(t,x) \nu_0 \llcorner \mathbf S^c = (p_{t,x})_\sharp \mu_0\llcorner \mathbf S^c = (p_{t,x})_\sharp \sigma_0\llcorner \mathbf S^c = \wt e(t,x) \xi_0\llcorner \mathbf S^c \qquad 
$$
and since $\nu_0, \xi_0$ are positive measures and $e, \wt e \in \{-1, 1\}$, we conclude
\begin{equation}
   \xi_0\llcorner \mathbf S^c = \nu_0 \llcorner \mathbf S^c \quad \text{and} \quad  e(t,x) = \wt e(t,x) \qquad \text{for $\nu_0$-a.e. and $\xi_0$-a.e.}
\end{equation}
We deduced that $\mu_0, \sigma_0$ restricted to $\mathbf S^c$ write in the form
\begin{equation}\label{eq:sourceform}
    \mu_0 \llcorner \mathbf S^c = e(t,x) \delta_{\bar z(t,x)} \otimes \mathbf j, \qquad \sigma_0 \llcorner \mathbf S^c = e(t,x) \delta_{\bar w(t,x)} \otimes \mathbf j
\end{equation}
where $\mathbf j := \nu_0\llcorner \mathbf S^c = \xi_0\llcorner \mathbf S^c$.
Using also the elementary identity 
$$
\frac{m(t,x)^2}{\rho(t,x)} + \frac{\rho(t,x)^3}{12} = \int_{w(t,x)}^{z(t,x)}v^2 \dif v = \int_{\mathbb R} v^2 (\chi(t,x,v) - \psi(t,x,v)) \dif v
$$
we deduce, by multiplying \eqref{eq:zchi}, \eqref{eq:wchi} by $v$ and subtracting \eqref{eq:wchi} from \eqref{eq:zchi}, that 
\begin{equation}\label{eq:momentum}
    0 = \partial_t m + \partial_x \Big(\frac{m^2}{\rho} + \frac{\rho^3}{3}\Big) = (p_{t,x})_\sharp (\sigma_1 - \mu_1) + \int  v \dif (\mu_0- \sigma_0)(\cdot,\cdot, v)
\end{equation}
where  we denote by $\int  v \dif (\mu_0- \sigma_0)(\cdot,\cdot, v)$ the measure on $(t,x)$ defined by 
$$
\iint \varphi(t,x) \dif \Big[ \int  v \dif (\mu_0- \sigma_0)(\cdot,\cdot, v) \Big](t,x) := \iiint \varphi(t,x) v \dif \mu_0 - \iiint \varphi(t,x) v \dif \sigma_0
$$
for all $\varphi:[0,T] \times \mathbb R \to \mathbb R$ continuous with compact support.
Since $\sigma_1 \llcorner\mathbf S^c = 0 = \mu_1 \llcorner \mathbf S^c$, we deduce from \eqref{eq:momentum} that
\begin{equation}
   0 =  \int  v \dif (\mu_0- \sigma_0)(\cdot,\cdot, v) \llcorner \mathbf S^c 
\end{equation}
which, using the expression \eqref{eq:sourceform}, means
\begin{equation}
    \bar z(t,x)  \mathbf j =  \bar w(t,x)\mathbf j \qquad \text{for $\mathbf j$-a.e. $(t,x)$.}
\end{equation}
But this just means $\bar w(t,x) = \bar z(t,x)$ for $\mathbf j$-a.e. $(t,x)$. Therefore 
$$
\partial_{vv} m\llcorner \mathbf S^c = \partial_v (\mu_1 - \sigma_1)\llcorner \mathbf S^c  + (\mu_0 - \sigma_0)\llcorner \mathbf S^c = 0.
$$
Since $m$ is compactly supported in $v$, this implies $m = 0$ in $\mathbf S^c$ and this concludes the proof.
\end{proof}
We now analyze the shocks structure in order to obtain more information about the dissipation measures.

\begin{defi}
    A pair of states $\bs u^-, \bs u^+ \in \mc K$  is an \emph{admissible shock} if 
    \begin{itemize}
        \item The Rankine-Hugoniot conditions hold:
          \begin{equation}\label{eq:RHiso}
        \begin{aligned}
        & \frac{z_+^2}{2}-\frac{z_-^2}{2} - \sigma (z_+ - z_-) =       \frac{w_+^2}{2}-\frac{w_-^2}{2} - \sigma (w_+ - w_-)  \\
        & \frac{z_+^3}{3}-\frac{z_-^3}{3} - \sigma \Big(\frac{z_+^2}{2} - \frac{z_-^2}{2}\Big) =       \frac{w_+^3}{3}-\frac{w_-^3}{3} - \sigma \Big(\frac{w_+^2}{2} - \frac{w_-^2}{2}\Big)
        \end{aligned}
    \end{equation}
    \item  The energy dissipation holds:
    \begin{equation}\label{eq:dissiso}
        d_E:= \frac{z_+^4}{4}-\frac{z_-^4}{4} - \sigma \Big(\frac{z_+^3}{3} - \frac{z_-^3}{3}\Big) -       \frac{w_+^4}{4}-\frac{w_-^4}{4} - \sigma \Big(\frac{w_+^3}{3} - \frac{w_-^3}{3}\Big) \leq 0
    \end{equation}
    \end{itemize}
\end{defi}

By the same arguments of \cite{DOW03} we deduce that $\mathscr H^1$-a.e. point $(t,x) \in \mathbf S$ is an admissible shock, where $\bs u^-, \bs u^+$ are the left and right traces on $\mathbf S$.

The proof of the following lemma is classical, and is a special case of the second order contact between the rarefaction and shock curve, together with the fact that the entropy dissipation for genuinely nonlinear systems is cubic with respect to the shock strength; we refer to \cite[Chapter 5]{Bbook} for a proof of this well known facts.
\begin{lemma}\label{lemma:shockstru}
If $\bs u^-, \bs u^+$ is an admissible shock, then there is a positive constant $C > 0$ such that one of the following holds:
\begin{itemize}
    \item[(a)] Shocks of the first family: $w_+ < w_-$, and 
    $$
    |z_+-z_-| \leq C\cdot (w_+-w_-)^3, \qquad \left|\sigma - \frac{w_++w_-}{2}\right| \leq C\cdot (w_+-w_-)^3
    $$
    $$
    d_E(w^+, w^-, z^+, z^-) \leq -C \cdot |w_+-w_-|^3
    $$
    \item[(b)] Shocks of the second family: $z_+ < z_-$, and 
    $$
    |w_+-w_-| \leq C\cdot (z_--z_+)^3, \qquad \left|\sigma - \frac{z_++z_-}{2}\right| \leq C\cdot (z_--z_+)^3
    $$
    $$
    d_E(w^+, w^-, z^+, z^-) \leq - C \cdot|z_+-z_-|^3
    $$
    \end{itemize}
\end{lemma}

 Next, we use the structure of the shock to deduce that it is possible to choose the measures  $\mu_1, \sigma_1$ in \eqref{eq:zchi}, \eqref{eq:wchi} with a definite sign.

\begin{lemma}\label{lemma:signedinv}
    Let $\bs u \in \mathbf L^\infty([0,T] \times \mathbb R, \mathcal K)$ be an entropy solution of \eqref{eq:iso}. Then the Riemann invariants satisfy
    \begin{equation}\label{eq:zchi1}
        \partial_t \chi  + v \partial_x \chi = \mu_0 + \partial_v \mu_1, \qquad \chi(t,x,v) := \begin{cases}
            1 & \text{if $ v \leq z(t,x)$}\\
            0 & \text{otherwise}
        \end{cases}
    \end{equation}
    \begin{equation}\label{eq:wchi1}
        \partial_t \psi + v \partial_x \psi = \sigma_0 + \partial_v \sigma_1, \qquad \psi(t,x,v) := \begin{cases}
            1 & \text{if $v \leq w(t,x)$}\\
            0 & \text{otherwise}
            \end{cases}
    \end{equation}
    where $\mu_1, \sigma_1$ are signed measures.
\end{lemma}
\begin{proof}
    For almost every point $(t,x) \in \mathbf S$, \eqref{eq:RHiso} and \eqref{eq:dissiso} are satisfied, where $z_+, z_-$ and $w_+, w_-$ are the left and right traces on $\mathbf S$, which exist thanks to Theorems \ref{thm:dimdthm}, \ref{prop:wzfes}.
By (3) of Theorem \ref{thm:dimdthm} we can characterize the distributions on the right hand side of \eqref{eq:zchi}, \eqref{eq:wchi}. 
\begin{equation}
    \langle \mu_0 \llcorner \mathbf J_w + \partial_v \mu_1 \llcorner \mathbf J_w, \eta^\prime \varphi\rangle  =  \int_{[0,T] \times \mathbb R} \int_{w_-}^{w_+} \eta^\prime(v)(v- \sigma) \dif v \dif \mathscr H^1 \llcorner \mathbf J_w,
\end{equation}
\begin{equation}
    \langle \sigma_0 \llcorner \mathbf J_z + \partial_v \sigma_1 \llcorner \mathbf J_z, \eta^\prime \varphi\rangle  =  \int_{[0,T] \times \mathbb R} \int_{z_-}^{z_+} \eta^\prime(v)(v- \sigma) \dif v \dif \mathscr H^1 \llcorner \mathbf J_z.
\end{equation}
Here as in the proof of Theorem \ref{thm:isereg} $\mu_i, \sigma_i$ are chosen to be minimal with respect to the relation $\preceq$ and $\mathbf J_w, \mathbf J_z$ are the corresponding jump sets \eqref{eq:setJ}.
We claim that up to an $\mathscr H^1$-negligible set there holds
\begin{equation}\label{eq:SJwJz}
\mathbf S = \mathbf J_w = \mathbf J_z.
\end{equation}
In fact, by definition of $\mathbf S$ we have $\mathbf J_w \cup \mathbf J_z = \mathbf S$. Since both $w, z$ admit left and right strong traces on $\mathbf S$ (see \cite{Va01}), we deduce that on $\mathscr H^1$-almost every point of $\mathbf S$ the Rankine-Hugoniot conditions \eqref{eq:RHiso} hold. In particular, the following holds:
\begin{equation}\label{eq:wiffz}
w^-(x) \neq w^+(x) \quad \Leftrightarrow \quad z^-(x) \neq z^+(x) \quad \text{for $\mathscr H^1$-a.e. $x \in \mathbf S$}.
\end{equation}
By the representation formula (3) of Theorem \ref{thm:dimdthm}, we deduce $w^- \neq w^+$ for $\mathscr H^1$-a.e. in $\mathbf J_w$, which then implies by \eqref{eq:wiffz} that $z^- \neq z^+$ for $\mathscr H^1$-a.e. in $\mathbf J_w$. But then $\mathscr H^1$-a.e. point $x \in \mathbf J_w$ cannot be a vanishing mean oscillation points for $z$, and therefore by Theorem \ref{thm:dimdthm} they are contained in $\mathbf J_z$ up to $\mathscr H^1$-negligible subsets: this shows $\mathbf J_w \subset \mathbf J_z$. The converse inclusion is proved symmetrically, and this ultimately proves the claim \eqref{eq:SJwJz}.

We partition $\mathbf S$ into the 1-shocks $\mathbf S^1$ and 2-shocks $\mathbf S^2$ accordingly to Lemma \ref{lemma:shockstru}.
In particular, say for the $1$-shocks, we have 
$$
\begin{aligned}
\mu_0 \llcorner \mathbf S^1 + \partial_v \mu_1 \llcorner \mathbf S^1 & = \mathbf 1_{(w^+, w^-)}(v) (v-\sigma) \mathscr H^1 \llcorner \mathbf S^1 \\
& = \mathbf 1_{(w^+, w^-)}(v) (v-\bar \sigma) \mathscr H^1 \llcorner \mathbf S^1  +\mathbf 1_{(w^+, w^-)}(v) (\bar \sigma- \sigma) \mathscr H^1 \llcorner \mathbf S^1 \\
& = : \partial_v \wt \mu_1 + \wt \mu_0
\end{aligned}
$$
where $\bar \sigma = (w^+ + w^-)/2$ and 
$$
\wt \mu_1 = \mathbf 1_{(w^+, w^-)}(v) (f(v)- f(w^+) - \bar \sigma(v- w^+)) \mathscr H^1 \llcorner \mathbf S^1 
$$
where $f(v) = v^2/2$. Moreover
$$
\begin{aligned}
    \sigma_0\llcorner \mathbf S^1  + \partial_v \sigma_1 \llcorner \mathbf S^1 & = \mathbf 1_{(z^-, z^+)}(v)(v-\sigma) \mathscr H^1 \llcorner \mathbf S^1 =: \wt \sigma_0
\end{aligned}
$$
Lemma \ref{lemma:shockstru} immediately implies that $\wt \mu_1$ has a positive sign, and that $\wt \mu_0, \wt \sigma_0$ are locally finite measures since they  satisfy, for some constant $C> 0$, 
$$
|\wt \mu_0|, |\wt \sigma_0| \leq C |\mu_{\eta_E}| \qquad \text{as measures}
$$
where $\eta_E$ is the energy dissipation.
Therefore we have on 1-shocks that 
$$
\mu_0 \llcorner \mathbf S^1 + \partial_v \mu_1 \llcorner \mathbf S^1 = \partial_v \wt \mu_1 + \wt \mu_0, \qquad \sigma_0 \llcorner \mathbf S^1 + \partial_v \sigma_1 \llcorner \mathbf S^1 = \wt \sigma_0
$$
where $\wt \mu_1$ is a positive measure, and $\wt \mu_0, \wt \sigma_0$ are locally finite measures.  An entirely symmetric decomposition holds for 2-shocks, and this concludes the proof.
\end{proof}

\begin{coro}
    Let $\bs u =(\rho, m) \in \mathbf L^\infty([0,T] \times \mathbb R; \mathcal K)$ be an entropy solution of \eqref{eq:iso}. Then $(\rho, m) \in B^{1/2, 1}_{\infty, loc}$.
\end{coro}
\begin{proof}
    The proof follows immediately by combining Theorem \ref{thm:besov} and Lemma \ref{lemma:signedinv}.
\end{proof}

\noindent {\bf Acknowledgements.}  
F.~Ancona, E.~Marconi and L.~Talamini are members of GNAMPA of the ``Istituto Nazionale di Alta Matematica
F.~Severi". F.~Ancona and E.~Marconi are partially supported by the PRIN
Project 20204NT8W4 “Nonlinear evolution PDEs, fluid dynamics and transport
equations: theoretical foundations and applications”,
and by the PRIN 2022 PNRR Project P2022XJ9SX ``Heterogeneity on the road - modeling, analysis, control''.
E.~Marconi  is also partially supported by H2020-MSCA-IF “A~Lagrangian approach: from conservation laws to line-energy Ginzburg-Landau models”. Luca Talamini is partially supported by the GNAMPA - Indam Project 2025 \say{Rappresentazione lagrangiana per sistemi di leggi di conservazione ed equazioni cinetiche.}

\vspace{0.5cm}
\noindent {\bf Declaration of interests.}  
The authors report there are no competing interests to declare.

\end{document}